\documentclass[leqno,10.5pt]{article} 
\usepackage{amsmath, amssymb}  
\usepackage{amsthm} 
\usepackage{amscd}   
\usepackage{amsmath, amssymb}    
\usepackage{amsthm} 
\usepackage{ascmac}
\usepackage{amscd}      
\usepackage{color}

\theoremstyle{plain} 
\newtheorem{theorem}{\indent\sc Theorem}[section]
\newtheorem{lemma}[theorem]{\indent\sc Lemma}
\newtheorem{corollary}[theorem]{\indent\sc Corollary}
\newtheorem{proposition}[theorem]{\indent\sc Proposition}

\theoremstyle{definition} 
\newtheorem{definition}[theorem]{\indent\sc Definition}

\newtheorem{remark}[theorem]{\indent\sc Remark}
\newtheorem{example}[theorem]{\indent\sc Example}

\newcommand{\C}{\mathbb{C}}
\newcommand{\R}{\mathbb{R}}
\newcommand{\Q}{\mathbb{Q}}

\newcommand{\Z}{\mathbb{Z}}

\newcommand{\N}{\mathbb{N}}

\def\2{I\hspace{-.1em}I}

\title{Rodrigues formula and linear independence for values \\of hypergeometric functions with parameters vary}
\author{\textsc{Makoto Kawashima}}
\date{ }   
\begin{document}

\maketitle
 


\begin{abstract}
In this article, we prove a generalized Rodrigues formula for a wide class of holonomic Laurent
series, which yields a new linear independence criterion concerning their values at algebraic points.
This generalization yields a new construction of Padé approximations including those for Gauss
hypergeometric functions. In particular, we obtain a linear independence criterion over a number field
concerning values of Gauss hypergeometric functions, allowing \emph{the parameters of Gauss hypergeometric functions} to vary.
\end{abstract}

\section{Introduction}
We give here a linear independence criterion for values over number fields, by using Pad\'e approximation,
for a certain class of holonomic Laurent series with algebraic coefficients.

As a consequence, over a number field we show a linear independence criterion of
values of Gauss hypergeometric functions, where we let the parameters vary, which is the novel part.

The Pad\'e approximation has been appeared as one of major methods in Diophantine problems since the works of Hermite and Pad\'e \cite{Pade1, Pade2}.
To solve number theoretical program by the Pad\'e approximation, we usually need to construct a system of Pad\'e approximants in an explicit form.
Pad\'e approximants can be constructed by linear algebra with estimates using Siegel's lemma via Dirichlet's box principle. 
However, it is not {{always}} enough to establish arithmetic applications such as the linear independence criterion. Indeed, we are obliged to explicitly 
construct Pad\'e approximants to provide sufficiently sharp estimates instead. In general, it is known that this step can be performed for specific functions only.

In this article, we succeed in proving a generalized Rodrigues formula, which gives an explicit construction of Pad\'{e} approximations
for a new and wide class of holonomic Laurent series. We introduce a linear map $\varphi_f$ (see {{Equation \eqref{varphi f}}}) with respect to
a given holonomic Laurent series $f(z)$, which describes a necessary and sufficient condition to explicitly construct Pad\'e approximants 
by studying 
${\rm{ker}}\, \varphi_f$. We state necessary properties of ${\rm{ker}}\, \varphi_f$ by looking at related differential operators.

The construction of Pad\'{e} approximants for Laurent series dates back to the classical works of Legendre and Rodrigues.
In $1782$, Legendre discovered a system of orthogonal polynomials the so-called Legendre polynomials.
In $1816$, Rodrigues established a simple expression for the Legendre polynomials, called Rodrigues formula by Hermite.
See \cite{A}, where Askey described a short history of the Rodrigues formula.
It is known that Legendre polynomials provide Pad\'{e} approximants of the logarithmic function. 
After Legendre and Rodrigues, various kinds of Pad\'{e} approximants of Laurent series have been developed
by Rasala \cite{R}, Aptekarev \emph{et al.} Assche \cite{A-B-V}, Rivoal \cite{Rivoal} and Sorokin \cite{S1,S2,S3}.
We note that Alladi and Robinson \cite{A-R}, also Beukers \cite{B1,B2,B3} applied the Legendre polynomials to solve central irrationality questions, 
and many results were shown in the following papers by Rhin and Toffin \cite{R-T}, Hata \cite{H1,H2,H3} and Marcovecchio \cite{M}.
The author together with S.~David and N.~Hirata-Kohno \cite{DHK1,DHK2,DHK3,DHK4} also proved the linear independence criterion 
concerning certain specific functions in a different setting.

By trying a new approach, distinct from those in \cite{DHK4}, the author shows how to construct
new generalized Pad\'{e} approximants of Laurent series. This method allows to provide a linear independence criterion
for the Gauss hypergeometric functions, letting the parameters vary.
The case has not been considered among known results before, although Gauss hypergeometric functions is a well-known classical function.

The approach relies on the linear map ${\varphi_f}$ ({see} {{Equation \eqref{varphi f}}}) to construct the Pad\'{e} approximants
in \emph{an explicit but formal manner}. This idea has been partly used but in a different expression in \cite{DHK1,DHK2,DHK3,DHK4}, as well as in \cite{KP} by  Po\"{e}ls and the author. 

The main point in this article is that we re-describe Rodrigues formula itself from a formal point of view to find suitable differential operators that enable us to construct Pad\'{e} approximants themselves, \emph{instead of Pad\'{e}-type approximants}.
This part is done for the functions
whose Pad\'{e} approximants have never been explicitly given before.

Consequently, our corollary provides arithmetic applications, for example,
the linear independence of the concerned values at different points
for a wider class of functions, which was not achieved in \cite{A-B-V}.

In the first part of this article, we discuss an explicit construction of Pad\'{e} approximants.
Our final aim is to find a general method to explicitly obtain Pad\'{e} approximants for given Laurent series.
Here, we partly succeed in giving a solution to this fundamental question on the Rodrigues formula for specific Laurent series that
can be transformed to polynomials by the differential operator of order $1$.
Precisely speaking, we indeed generalize the Rodrigues formula to a new class of holonomic series ({see} Theorem $\ref{main general}$).

In the second part, we apply our explicit Pad\'{e} approximants of holonomic Laurent series for the linear independence problems of their values. 
As a corollary, we show below a new linear independence criterion for values of Gauss hypergeometric function, letting the parameters vary.
We recall the Gauss hypergeometric function.
For a rational number $x$ and a nonnegative integer $k$, we denote the $k$ th Pochhammer symbol: $(x)_0=1$, $(x)_k=x(x+1)\cdots (x+k-1)$.
For $a,b,c\in \Q$ that are not negative integers, we define 
$$_{2}F_1(a,b,c\,|z)=\sum_{k=0}^{\infty}\dfrac{(a)_k(b)_k}{(c)_kk!}z^{k}.$$
We can now state the following theorem.
\begin{theorem} \label{main!}
Let $u,\alpha$ be integers with $u\ge 2$ and $|\alpha| \ge 2$.
Assume $$V(\alpha):=\log |\alpha|-\log 2-\left(2-\dfrac{1}{u}\right)\left(\log u+\sum_{\substack{q:\mathrm{prime} \\ q|u}}\dfrac{\log q}{q-1}\right)-\dfrac{u-1}{\varphi(u)}>0,$$
where $\varphi$ is Euler's totient function.
Then the real numbers:
$$1,  {}_2F_1\left(\frac{1+l}{u},1, \left.\frac{u+l}{u}\,\right|\frac{1}{\alpha^u}\right) \ \ \ (0\le l \le u-2)$$
are linearly independent over $\Q$. 
\end{theorem}
The following table gives suitable data for $u$ and $\alpha$ so as to ensure $V(\alpha)>0$:
$$
\begin{array}{|c| c |c |c | c|c| c| c| c| c | c| c| c| c| c|}
\hline
  u & 2 & 3 & 4 & 5 & 6 & 7 & 8 & 9 & 10 & 11 & 12 & 13 & 14 & 15\\
\hline
 \alpha \ge & e^{3.78} & e^{4.44} & e^{5.84} & e^{5.32} & e^{8.76} & e^{5.91} & e^{7.65} & e^{7.22} & e^{9.40} & e^{6.73} & e^{10.59} & e^{7.04} & e^{9.92} & e^{9.52}\\
\hline
\end{array}
$$

The present article is organized as follows. 
In Section $2$, we collect basic notions and recall the Pad\'{e}-type approximants of Laurent series.
To achieve an explicit construction of Pad\'{e} approximants, which is of particular interest, we introduce a morphism $\varphi_f$ 
associated with a Laurent series $f(z)$. 
To analyze the structure of ${\rm{ker}}\, \varphi_f$ is a crucial point for our program (see Proposition \ref{equivalence Pade}).  
Indeed, we provide a proper subspace, in some cases this is the whole space, of ${\rm{ker}}\, \varphi_f$ derived from the differential operator that annihilates $f$ (see Corollary $\ref{inc}$).
This is the key ingredient required to generalize the Rodrigues formula.

In Section $3$, we introduce the weighted Rodrigues operator, which is firstly defined in \cite{A-B-V} as well as basic properties that are going to be used in the course of the proof. 

In Section $4$, we give a generalization of the Rodrigues formula to Pad\'{e} approximants of certain holonomic series, by using the weighted Rodrigues operators (see Theorem $\ref{main general}$).
In Section $5$, we introduce the determinants associated with the Pad\'{e} approximants obtained in Theorem $\ref{main general}$. 
To prove the nonvanishing of these determinants is one of the most crucial step to obtain irrationality as well as linear independence results.  
We discuss some examples of Theorem $\ref{main general}$ and Proposition $\ref{det}$ in Section $6$.
Example $\ref{Chebyshev}$ is the particular example concerning Theorem $\ref{main!}$.  
In Section $7$, we state the more precise theorem than Theorem $\ref{main!}$ (see Theorem $\ref{special hypergeometric}$). This section is devoted to the proof of Theorem $\ref{special hypergeometric}$.
{{The Appendix is devoted  to describing a result due to Fischler and Rivoal in \cite{F-R}.
They gave a condition on the differential operator of order $1$ with polynomial coefficients so as to be a $G$-operator. 
Indeed, this result is crucial to apply Theorem $\ref{main general}$ to $G$-functions. More precisely,  whenever the operator is a $G$-operator,
then the Laurent series considered in Theorem $\ref{special hypergeometric}$
turn out to be $G$-functions.}}

\section{Pad\'{e}-type approximants of Laurent series}
Throughout this section, we fix a field $K$ of characteristic $0$.
{{We denote the formal power series ring of variable $1/z$ with coefficients $K$ by $K[[1/z]]$ and the field of fractions by $K((1/z))$. We say an element of $K((1/z))$ is a formal Laurent series.}}
We define the order function at $z=\infty$ by
$${\rm{ord}}_{\infty}:K((1/z)) \longrightarrow \Z\cup \{\infty\}; \ \sum_{k} \dfrac{a_k}{z^k} \mapsto \min\{k\in\Z\cup \{\infty\} \mid a_k\neq 0\}.$$
{{Note that, for $f\in K((1/z))$, ${\rm{ord}}_{\infty} \, f=\infty$ if and only if $f=0$.}}
We recall without proof the following elementary fact :
\begin{lemma} \label{pade}
Let $m$ be a nonnegative integer, $f_1(z),\ldots,f_m(z)\in (1/z)\cdot K[[1/z]]$ and $\boldsymbol{n}=(n_1,\ldots,n_m)\in \N^{m}$.
Put $N=\sum_{j=1}^mn_j$.
For a nonnegative integer $M$ with $M\ge N$, there exist polynomials $(P,Q_{1},\ldots,Q_m)\in K[z]^{m+1}\setminus\{\boldsymbol{0}\}$ satisfying the following conditions:
\begin{align*} 
&(i) \ {\rm{deg}}P\le M,\\
&(ii) \ {\rm{ord}}_{\infty} \left(P(z)f_j(z)-Q_j(z)\right)\ge n_j+1 \ \ \text{for} \ \ 1\le j \le m.
\end{align*}
\end{lemma}
\begin{definition}
We say that a vector of polynomials $(P,Q_{1},\ldots,Q_m) \in K[z]^{m+1}$ satisfying properties $(i)$ and $(ii)$ {{is}} a weight $\boldsymbol{n}$ and degree $M$ Pad\'{e}-type approximant of  $(f_1,\ldots,f_m)$.
For such approximants $(P,Q_{1},\ldots,Q_m)$ of $(f_1,\ldots,f_m)$, we call the formal Laurent series $(P(z)f_j(z)-Q_{j}(z))_{1\le j \le m}$, that is to say remainders, as weight $\boldsymbol{n}$ degree $M$ Pad\'{e}-type approximations of $(f_1,\ldots,f_m)$.
\end{definition}

Let $f(z)=\sum_{k=0}^{\infty} f_k/z^{k+1}\in (1/z)\cdot K[[1/z]]$.
We define a {{$K$-linear map}} $\varphi_f\in {\rm{Hom}}_K(K[t],K)$ by
\begin{align} \label{varphi f}
\varphi_f:K[t]\longrightarrow K; \ \ \ t^k\mapsto f_k \ \ \ (k\ge0) .
\end{align}

The above {{linear map}} extends naturally a $K[z]$-{{linear map}} $\varphi_f: K[z,t]\rightarrow K[z]$, and then to a $K[z][[1/z]]$-{{linear map}} $\varphi_f: K[z,t][[1/z]]\rightarrow K[z][[1/z]]$. With this notation, the formal Laurent series $f(z)$ satisfies the following crucial identities {{(see \cite[$(6.2)$ page 60 and $(5.7)$ page 52]{N-S})}}:
\begin{align*}
&f(z)=\varphi_f \left(\dfrac{1}{z-t}\right),\ \ \ P(z)f(z)-\varphi_f\left(\dfrac{P(z)-P(t)}{z-t}\right)\in (1/z)\cdot K[[1/z]] \ \ \text{for any} \ \ P(z)\in K[z].
\end{align*}
\begin{lemma} \label{equivalence Pade}
Let $m$ be a nonnegative integer, $f_1(z),\ldots,f_m(z)\in (1/z)\cdot K[[1/z]]$ and $\boldsymbol{n}=(n_1,\ldots,n_m)\in \N^m$. Let $M$ be a positive integer and $P(z)\in K[z]$ a nonzero polynomial with 
$M\ge \sum_{j=1}^mn_j$ and ${\rm{deg}}\,P\le M$. Put $Q_j(z)=\varphi_{f_j}\left((P(z)-P(t))/(z-t)\right)\in K[z]$ for $1\le j \le m$.

Then the following are equivalent.

$(i)$ The vector of polynomials $(P,Q_1,\ldots,Q_m)$ is a weight $\boldsymbol{n}$ Pad\'{e}-type approximants of $(f_1,\ldots,f_m)$.

$(ii)$ We have $t^kP(t)\in {\rm{ker}}\,\varphi_{f_j}$ for $1\le j \le m$, $0\le k \le n_j-1$.
\end{lemma}
\begin{proof} 
By the definition of $Q_j(z)$,
$$
P(z)f_j(z)-Q_j(z)=\varphi_{f_j}\left(\dfrac{P(t)}{z-t}\right)\in (1/z)\cdot K[[1/z]].
$$
Above equality yields that the vector of polynomials $(P,Q_1,\ldots,Q_m)$ being a weight $\boldsymbol{n}$ Pad\'{e}-type approximants of $(f_1,\ldots,f_m)$ is equivalent to the order of the Laurent series 
$$\varphi_{f_j}\left(\dfrac{P(t)}{z-t}\right)={\sum_{k=0}^{\infty}}\dfrac{\varphi_{f_j}\left(t^kP(t)\right)}{z^{k+1}}$$ being greater than or equal to $n_j+1$ for $1\le j \le m$. This shows the equivalence of items $(i)$ and $(ii)$. 
\end{proof}
Lemma \ref{equivalence Pade} indicates that it is useful to study ${\rm{ker}}\, \varphi_f$ for the explicit construction of Pad\'{e}-type approximants of Laurent series.  
We are now going to investigate ${\rm{ker}}\, \varphi_{f}$ for a holonomic Laurent series $f\in (1/z)\cdot K[[1/z]]$. 
We denote the differential operator $\tfrac{d}{dz}$ (resp. $\tfrac{d}{dt}$) by $\partial_z$ (resp. $\partial_t$). 
We describe the action of a differential operator $D$ on a function $f$ by $D\cdot f$ and denote $\partial_z \cdot f$ by $f^{\prime}$.

To begin with, let us introduce a map 
\begin{align}
\iota :K(z)[\partial_z]\longrightarrow K(t)[\partial_t]; \ \ \sum_j P_j(z)\partial^j_z \mapsto \sum_j (-1)^j \partial^j_t P_j(t). \label{iota diff}
\end{align}
{{Note, for $D\in K(z)[\partial_z]$, $\iota(D)$ is called the adjoint of $D$ and relates to the dual of differential module $K(z)[\partial_z]/K(z)[\partial_z]D$ (see \cite[III Exercises $3)$]{An1}).}}
For $D\in K(z)[\partial_z]$, we denote $\iota(D)$ by $D^{*}$. 
Notice that we have $(DE)^{*}=E^*D^{*}$ for any $D,E\in K(z)[\partial_z]$.  
\begin{lemma} \label{lemma 1}
For $D\in K[z,\partial_z]$, there exists a polynomial $P(t,z)\in K[t,z]$ satisfying $$D\cdot \dfrac{1}{z-t}=P(t,z)+D^{*}\cdot \dfrac{1}{z-t}.$$
\end{lemma} 
\begin{proof}
Let $m,n$ be nonnegative integers. It suffices to prove the case $D=z^m\partial^n_z$. 
Then,
\begin{align} \label{cal 1}
D \cdot \dfrac{1}{z-t}&=\dfrac{(-1)^nn!z^m}{(z-t)^{n+1}}=(-1)^n\sum_{k=0}^{\infty}\dfrac{(n+k)!}{k!}\dfrac{t^k}{z^{k+1+n-m}}.
\end{align}
We define a polynomial $P(t,z)$ by $0$ if $m\le n$ and 
$$P(t,z)=
(-1)^n {\displaystyle\sum_{k=0}^{m-n-1}}\dfrac{(n+k)!}{k!}t^kz^{m-n-k-1}
$$
for $m>n$. 
Equation $(\ref{cal 1})$ implies   
\begin{align*}
D \cdot \dfrac{1}{z-t}-P(t,z)&=(-1)^n\sum_{k=\max\{m-n,0\}}^{\infty}\dfrac{(n+k)!}{k!}\dfrac{t^k}{z^{k+1+n-m}}\\
                           &=(-1)^n\sum_{k=0}^{\infty}(k+1+m-n)\cdots(m+k)\dfrac{t^{k+m-n}}{z^{k+1}}.
\end{align*}
However,
\begin{align*} 
D^*\cdot \dfrac{1}{z-t}&=(-1)^n\partial^n_t \cdot\dfrac{t^m}{z-t}=(-1)^n\sum_{k=0}^{\infty}\partial^n_t\cdot \dfrac{t^{m+k}}{z^{k+1}}\\
&=(-1)^n\sum_{k=0}^{\infty}(k+1+m-n)\cdots(m+k)\dfrac{t^{k+m-n}}{z^{k+1}}.
\end{align*}
The above equalities yield $$D\cdot \dfrac{1}{z-t}-P(t,z)=D^*\cdot \dfrac{1}{z-t}.$$ This completes the proof of Lemma $\ref{lemma 1}$.
\end{proof}
We introduce the projection morphism $\pi$ by
\begin{align*} 
\pi :K[z][[1/z]]\longrightarrow K[z][[1/z]]/K[z]\cong (1/z)\cdot K[[1/z]]; \ \ \ f(z)=P(z)+\tilde{f}(z)\mapsto \tilde{f}(z), 
\end{align*}
where $P(z)\in K[z]$ and $\tilde{f}(z)\in (1/z)\cdot K[[1/z]]$.
Lemma $\ref{lemma 1}$ allows us to show the following key proposition.  
\begin{proposition} \label{key prop}
Let $D\in K[z,\partial_z]$ and $f(z)\in (1/z)\cdot K[[1/z]]$. We have $\varphi_{\pi(D\cdot f)}=\varphi_f\circ D^{*}$.
\end{proposition}
\begin{proof}
{{First, since $\varphi_f$ acts only on the parameter $t$, 
$$D\cdot f=D\circ \varphi_f\left(\dfrac{1}{z-t}\right)=\varphi_f \left(D\cdot \dfrac{1}{z-t}\right).$$}}
Lemma \ref{lemma 1} implies that there exists a polynomial $P(z)$ with 
$$D\cdot f=P(z)+\varphi_f\left(D^{*}\cdot \dfrac{1}{z-t} \right)=P(z)+\sum_{k=0}^{\infty}\dfrac{\varphi_f(D^{*}\cdot t^k)}{z^{k+1}}.$$
{{Note that $P(z)=\varphi_f(P(t,z))$ where $P(t,z)\in K[t,z]$ is defined in Lemma $\ref{lemma 1}$.}}
This shows that $\pi(D\cdot f)=\sum_{k=0}^{\infty}\varphi_f(D^{*}\cdot t^k)/{z^{k+1}}$ and therefore
$$\varphi_{\pi(D\cdot f)}(t^k)=\varphi_f\circ D^{*}(t^k) \ \ \ \text{for all} \ \ \ k\ge 0.$$
This concludes the proof of Proposition \ref{key prop}.
\end{proof} 
As a corollary of Proposition \ref{key prop}, the following crucial equivalence relations hold.
\begin{corollary} \label{inc} 
Let $f(z)\in (1/z)\cdot K[[1/z]]$ and $D\in K[z,\partial_z]$.

The following are equivalent:

$(i)$ $D\cdot f\in K[z]$;

$(ii)$ $D^{*}(K[t])\subseteq {\rm{ker}}\,\varphi_f$.
\end{corollary}
\begin{proof} 
Conditions $(i)$ and $(ii)$ are equivalent to $\pi(D\cdot f)=0$ and $\varphi_f \circ D^{*}=0$, respectively. Therefore by Proposition \ref{key prop}, we obtain the assertion. 
\end{proof}

\section{Weighted Rodrigues operators}
Let $K$ be a field of characteristic $0$.
Let us introduce the weighted Rodrigues {{operator}}, which is firstly defined by Aptekarev, Branquinho and Van Assche in \cite{A-B-V}.
\begin{definition} $($See \cite[Equation $(2.5)$]{A-B-V}$)$ 
Let $l\in \N$, $a_1(z),\ldots,a_l(z)\in K[z]\setminus\{0\}$, $b(z)\in K[z]$. Put $a(z)=a_1(z)\cdots a_l(z)$, $D=-a(z)\partial_z+b(z)$. 
For $n\in \N$ and a weight $\vec{r}=(r_1,\ldots,r_l)\in \Z^l$ with $r_i\ge0$, we define the weighted Rodrigues operator associated with $D$ by
$$R_{D,n,\vec{r}}=\dfrac{1}{n!}\left(\partial_z+\dfrac{b(z)}{a(z)}\right)^na(z)^n\prod_{v=1}^la_v(z)^{-r_v}\in K(z)[\partial_z].$$
In the case of $\vec{r}=(0,\ldots,0)$, we denote $R_{D,n,\vec{r}}=R_{D,n}$ and call this operator the $n^{{\rm{th}}}$ Rodrigues operator associated with $D$.

We denote the generalized {{Rodrigues}} operator associated with $D$ with respect to the parameter $t$ by 
$$\mathcal{R}_{D,n,\vec{r}}=\dfrac{1}{n!}\left(\partial_t+\dfrac{b(t)}{a(t)}\right)^na(t)^n\prod_{v=1}^la_v(t)^{-r_v} \in K(t)[\partial_t],$$
and $\mathcal{R}_{D,n,\vec{r}}=\mathcal{R}_{D,n}$ in the case of $\vec{r}=(0,\ldots,0)$.
\end{definition}
Let us show some basic properties of the weighted Rodrigues operator in order to obtain a generalization of Rodrigues formula of Pad\'{e} approximants of holonomic Laurent series.
In the following, for $a(z)\in K[z]$ (resp. $a(t)\in K[t]$), we denote the ideal of $K[z]$ (respectively $K[t]$), generated by $a(z)$ (respectively $a(t)$) by $(a(z))$ (respectively $(a(t))$).
\begin{proposition} \label{key 1 2}
Let $a(t),b(t)\in K[t]$ with $a(t)\neq 0$. Put $\mathcal{E}_{a,b}=\partial_t+b(t)/a(t)\in K(t)[\partial_t]$.

$(i)$ Let $n,k$ be nonnegative integers. Then there exists integers $(c_{n,k,l})_{0\le l  \le \min\{n,k\}}$ with $$c_{n,k,\min\{n,k\}}=(-1)^nk(k-1)\cdots (k-n+1),$$
\begin{align} \label{im eq}
t^k\mathcal{E}^n_{a,b}=\sum_{l=0}^{\min\{n,k\}}c_{n,k,l}\mathcal{E}^{n-l}_{a,b}t^{k-l}\in K(t)[\partial_t].
\end{align}

$(ii)$ Assume there exist polynomials $a_1(t),\ldots,a_l(t)\in K[t]$ with $a(t)=a_1(t)\cdots a_l(t)$. 
For an $l$-tuple of nonnegative integers $\boldsymbol{s}:=(s_1,\ldots,s_l)$, we denote by $\mathrm{I}(\boldsymbol{s})$ the ideal of $K[t]$ generated by ${\displaystyle\prod_{v=1}^l}a_v(t)^{s_v}$. Then for $n\ge 1$ and $F(t)\in \mathrm{I}(\boldsymbol{s})$, 
\begin{align} \label{contain ideal}
\mathcal{E}^n_{a,b}a(t)^n\cdot F(t)\in \mathrm{I}(\boldsymbol{s}).
\end{align}
\end{proposition}
\begin{proof}

$(i)$ We prove the assertion by induction on $(n,k)\in \Z^2$ with $n,k\ge0$. 
{{In the case of $n=0$ and any $k\ge 0$, the statement is trivial.}}
Let $n,k$ be nonnegative integers with $n\ge 1$ or $k\ge 1$. 
We assume that the assertion holds for any elements of the set
$(\tilde{n},\tilde{k})\in \{(\tilde{n},\tilde{k})\in \Z^2\mid 0\le \tilde{n},\tilde{k} \ \ \text{and} \ \  \tilde{n}< n \ \text{and} \ \tilde{k}\le k\}$. 
The equality $t^k\mathcal{E}_{a,b}=\mathcal{E}_{a,b}t^k-kt^{k-1}$ in $K[t,\partial_t]$ implies that we have
\begin{align}
t^k\mathcal{E}^n_{a,b}&=(\mathcal{E}_{a,b}t^k-kt^{k-1})\mathcal{E}^{n-1}_{a,b} \nonumber \\ 
                             &=\mathcal{E}_{a,b}\sum_{l=0}^{\min\{n-1,k\}}c_{n-1,k,l}\mathcal{E}^{n-1-l}_{a,b}t^{k-l}-k\sum_{l=0}^{\min\{n-1,k-1\}}c_{n,k-1,l}\mathcal{E}^{n-1-l}_{a,b}t^{k-1-l} \label{second eq a}\\
&=\sum_{l=0}^{\min\{n-1,k\}}c_{n-1,k,l}\mathcal{E}^{n-l}_{a,b}t^{k-l}-\sum_{l=0}^{\min\{n-1,k-1\}}kc_{n-1,k-1,l}\mathcal{E}^{n-1-l}_{a,b}t^{k-1-l}. \nonumber
\end{align}
Note that we use the induction hypothesis in line \eqref{second eq a}. This concludes the assertion for $(n,k)$.

$(ii)$ Let us prove the statement by induction on $n$. In the case of $n=1$, since 
\begin{align*}
\mathcal{E}_{a,b}a(t)\cdot F(t)&=(\partial_t a(t)+b(t))\cdot F(t)
                                        =a^{\prime}(t)F(t)+a(t)F^{\prime}(t)+b(t)F(t),
\end{align*}
using the Leibniz formula, we obtain Equation $(\ref{contain ideal})$. We assume Equation $(\ref{contain ideal})$ holds for $n\ge 1$. 
In the case of $n+1$, 
\begin{align} \label{ind n+1} 
\mathcal{E}^{n+1}_{a,b}a(t)^{n+1}\cdot F(t)&=\mathcal{E}_{a,b}\mathcal{E}^n_{a,b}a(t)^n\cdot a(t)F(t).
\end{align}
Note that we have $a(t)F(t)\in \mathrm{I}(\boldsymbol{s}+\boldsymbol{1})$ where $\boldsymbol{s}+\boldsymbol{1}:=(s_1+1,\ldots,s_d+1)\in \N^d$. 
Relying on the induction hypothesis, we deduce 
$\mathcal{E}^n_{a,b}a(t)^n\cdot a(t)F(t)\in {\mathrm{I}}(\boldsymbol{s}+\boldsymbol{1})$. Thus, there exists a polynomial $\tilde{F}(t)\in {\mathrm{I}}(\boldsymbol{s})$ with $\mathcal{E}^n_{a,b}a(t)^n\cdot a(t)F(t)=a(t)\tilde{F}(t)$.
Substituting this equality into Equation $(\ref{ind n+1})$, by using a similar argument to the case of $n=1$, we conclude $\mathcal{E}^{n+1}_{a,b}a(t)^{n+1}\cdot F(t)\in {\mathrm{I}}(\boldsymbol{s})$.
\end{proof}
\begin{corollary} \label{im cor} 
$(i)$ Let $a(z)\in K[z]\setminus\{0\}$ and $b(z)\in K[z]$. We put $D=-a(z)\partial_z+b(z)$. 
Let $f(z)\in (1/z)\cdot K[[1/z]]\setminus \{0\}$ with $D\cdot f(z)\in K[z]$. 
Put $\mathcal{E}_{a,b}=\partial_t+b(t)/a(t)\in K(t)[\partial_t]$. 
Then, for $n,k\in \Z$ with $0\le k<n$, we have $$t^k\mathcal{E}^n_{a,b}\cdot (a(t)^n)\subseteq {\rm{ker}}\,\varphi_f.$$ 

$(ii)$ Let $d,l \in \N$, $(n_1,\ldots,n_d)\in \N^d$ and $a_1(t),\ldots,a_l(t)\in K[t]\setminus \{0\}$.
Put $a(t)=a_1(t)\cdots a_l(t)$. For $b_1(t),\ldots,b_d(t)\in K[t]$ and $l$-tuple of nonnegative integers $\vec{r}_j=(r_{j,1},\ldots,r_{j,l})$ $(1\le j \le d)$, we put $D_j=-a(z)\partial_z+b_j(z)$ and 
$${{\mathcal{R}_{j,n_j}=\mathcal{R}_{D_j,n_j,\vec{r}_j}}}=\dfrac{1}{n_j!}\mathcal{E}^{n_j}_{a,b_j}a(t)^{n_j}\prod_{v=1}^la_v(t)^{-r_{j,v}}\in K(t)[\partial_t].$$ 
Let $s_1,\ldots,s_d$ be nonnegative integers and $F(t)\in \left({\displaystyle\prod_{v=1}^l}a_v(t)^{s_v+\sum_{j=1}^dr_{j,v}}\right)$. 
Then,
$$\prod_{j=1}^d\mathcal{R}_{j,n_j}\cdot F(t)\in \left({\displaystyle\prod_{v=1}^l}a_v(t)^{s_v}\right) .$$
\begin{center}{$($The statement in the first term holds for any order of product of operators $(\mathcal{R}_{j,n_j})_j$.$)$}\end{center}
\end{corollary}
\begin{proof}
$(i)$ By the definition of $D$, we have $D^{*}=\mathcal{E}_{a,b}a(t)$. Since we have $\mathcal{E}_{a,b}\cdot (a(t))\subseteq {\rm{ker}}\,\varphi_f$, by {{Corollary}} $\ref{inc}$, it suffices to show   
$t^k\mathcal{E}^n_{a,b}\cdot (a(t)^n)\subset \mathcal{E}_{a,b}\cdot (a(t))$. Relying on Proposition $\ref{key 1 2} \ (i)$, there are $\{c_{n,k,l}\}_{0\le l \le k}\subset {{\Z}}$ with 
\begin{align} \label{key 1 k<n}
t^k\mathcal{E}^n_{a,b}=\sum_{l=0}^k c_{n,k,l}\mathcal{E}^{n-l}_{a,b}t^{k-l}.
\end{align}
For an integer $l$ with $0\le l \le k$, 
\begin{align*} 
\mathcal{E}^{n-l}_{a,b}t^{k-l} \cdot (a(t)^n)\subset \mathcal{E}_{a,b}\mathcal{E}^{n-l-1}_{a,b}\cdot (a(t)^n).
\end{align*}
The Leibniz formula allows us to get $\mathcal{E}^{n-l-1}_{a,b}\cdot (a(t)^n)\subset (a(t))$.
Combining Equation $(\ref{key 1 k<n})$ and the above relation gives
$$t^k\mathcal{E}^n_{a,b}\cdot (a(t)^n)\subset \mathcal{E}_{a,b}\cdot (a(t)).$$
This completes the proof of item $(i)$.

$(ii)$  
It suffices to prove the assertion in the case of $d=1$. By the definition of $\mathcal{R}_{1,n_1}$,
\begin{align*}
\mathcal{R}_{1,n_1}\cdot F(t)&=\dfrac{1}{n_1!}\mathcal{E}^{n_1}_{a,b_1}a(t)^{n_1}\prod_{v=1}^{{l}} a_v(t)^{-r_{1,v}}\cdot F(t) \in \mathcal{E}^{n_1}_{a,b_1}a(t)^{n_1}\cdot \left(\prod_{v=1}^la_v(t)^{s_v}\right).
\end{align*}
Using Proposition $\ref{key 1 2} \ (ii)$, we conclude that 
$\mathcal{E}^{n_1}_{a,b_1}a(t)^{n_1}\cdot \left(\prod_{v=1}^la_v(t)^{s_v}\right)\subset  \left(\prod_{v=1}^la_v(t)^{s_v}\right)$. 
This completes the proof of item $(ii)$.
\end{proof}

\section{Rodrigues formula of Pad\'{e} approximants}
\begin{lemma} \label{solutions}
Let $a(z),b(z)\in K[z]$ with $a(z)\neq 0$, ${\rm{deg}}\, a=u$ and ${\rm{deg}}\,b=v$. Put $$D=-a(z)\partial_z+b(z)\in K[z,\partial_z], \ \ a(z)=\sum_{i=0}^ua_iz^i, \ \ b(z)=\sum_{j=0}^vb_jz^j,$$ and $w=\max\{u-2,v-1\}$.
Assume $w\ge 0$ {{and
\begin{align}\label{assump a,b}
a_u(k+u)+b_v\neq 0 \ \ \ \text{for all} \ \ \ k\ge0 \ \ \ \text{if} \ \ \  u-2=v-1.
\end{align}}}
Then there exist $f_0(z),\ldots,f_w(z)\in (1/z)\cdot K[[1/z]]$ that are linearly independent over $K$ and satisfy $D\cdot f_l(z)\in K[z]$ for $0\le l \le w$. 
\end{lemma} 
\begin{proof}
Let $f(z)=\sum_{k=0}^{\infty}f_k/z^{k+1}\in (1/z)\cdot K[[1/z]]$ be a Laurent series. 
There exists a polynomial $A(z)\in K[z]$ that depends on the operator $D$ and $f$ with ${\rm{deg}}\,A\le w$ and satisfying
\begin{align} \label{Df}
D\cdot f(z)=A(z)
+\sum_{k=0}^{\infty}\dfrac{\sum_{i=0}^ua_i(k+i)f_{k+i-1}+\sum_{j=0}^vb_jf_{k+j}}{z^{k+1}}.
\end{align}
Put $$\sum_{i=0}^ua_i(k+i)f_{k+i-1}+\sum_{j=0}^vb_jf_{k+j}=c_{k,0}f_{k-1}+\cdots+c_{k,w}f_{k+w}+c_{k,w+1}f_{k+w+1} \ \ \ \text{for} \ \ \ k\ge 0,$$ with $c_{0,0}=0$.
{{We remark that $c_{k,l}$ depends only on $a(z),b(z)$.}}  
Notice that $c_{k,w+1}$ is $a_u(k+u)$ if $u-2>v-1$, $b_v$ if $u-2<v-1$ and $a_u(k+u)+b_v$ if $u-2=v-1$.
{{Then by Equation \eqref{assump a,b}, we have $\min\{k\ge 0\mid c_{k',w+1}\neq 0 \ \text{for all} \ k'\ge k\}=0$
and thus the $K$-linear map:
$$K^{w+1}\longrightarrow \{f\in (1/z)\cdot K[[1/z]]\mid D\cdot f\in K[z]\}; \ \ (f_0,\ldots,f_w)\mapsto \sum_{k=0}^{\infty} \dfrac{f_k}{z^{k+1}},$$
where, for $k\ge w+1$, $f_k$ is determined inductively by 
\begin{align*}
&\sum_{i=0}^ua_i(k+i)f_{l,k+i-1}+\sum_{j=0}^vb_jf_{l,k+j}=0 \ \ \text{for} \ \ k\ge 0 
\end{align*}
is an isomorphism.}}
This completes the proof of Lemma $\ref{solutions}$.
\end{proof}
Let us state a generalization of the Rodrigues formula for Legendre polynomials to Pad\'{e} approximants of certain holonomic Laurent series, which gives a generalization of \cite[Theorem $1$]{A-B-V}.
{{In the following theorem, we construct Pad\'{e} approximants of the family of Laurent series considered in Lemma $\ref{solutions}$.}}
\begin{theorem} \label{main general}
Let $l,d\in \N$, $(a_1(z),\ldots,a_l(z))\in (K[z]\setminus\{0\})^l$ and $(b_{1}(z),\ldots,b_{d}(z))\in K[z]^d$.
Put $a(z)=a_{1}(z)\cdots a_{l}(z)$.
Put $D_{j}=-a(z)\partial_z+b_{j}(z) \in K[z,\partial_z]$ and $w_{j}=\max\{{\rm{deg}}\, a-2, {\rm{deg}}\, b_{j}-1\}$. Assume $w_j\ge 0$ for $1\le j \le d$ and \eqref{assump a,b} for $D_j$.
Let $f_{j,0}(z),\ldots,f_{j,w_{j}}(z)\in (1/z)\cdot K[[1/z]]$ be formal Laurent series that are linearly independent over $K$ satisfying $$D_{j}\cdot f_{j,u_{j}}(z)\in K[z] \ \ \text{for} \ \  0\le u_{j} \le w_{j}.$$
(The existence of such series is ensured by Lemma $\ref{solutions}$.)
Let $(n_1,\ldots,n_d)\in \N^d$. 
For an $l$-tuple of nonnegative integers $\vec{r}_j=(r_{j,1},\ldots,r_{j,l})$ $(1\le j \le d)$, we denote by $R_{j,n_j}$ the weighted Rodrigues operator $R_{D_{j},n_j,\vec{r}_j}$ associated with $D_j$. Assume
\begin{align}
&R_{{j_1},n_{j_1}}  R_{{j_2},n_{j_2}}=R_{{j_2},n_{j_2}} R_{{j_1},n_{j_1}} \ \ \text{for}  \ \ 1\le j_1,j_2\le d. \label{commutative}
\end{align}
Take a nonzero polynomial $F(z)$ that is contained in the ideal $(\prod_{v=1}^{l}a_{v}(z)^{\sum_{j=1}^d r_{j,v}})$ and put
\begin{align*}
&P(z)=\prod_{j=1}^dR_{j,n_j}\cdot F(z),\\
&Q_{j,u_{j}}(z)=\varphi_{f_{j,u_{j}}}\left(\dfrac{P(z)-P(t)}{z-t}\right)  \ \ \text{for} \ \ 1\le j \le d, \ 0\le u_{j}\le w_{j} .
\end{align*}
Assume $P(z)\neq 0$ (We need to assume $P(z)\neq 0$. For example, in the case of $d=1$, $D=-\partial_z z^2=-z^2\partial-2z$ and $n=1$, we have $P(z)=(\partial_z-2/z)z^2\cdot 1=0$.)
Then the vector of polynomials $(P(z),Q_{j,u_{j}}(z))_{\substack{ \\ 1\le j \le d \\ 0\le u_{j}\le w_{j}}}$ is a weight $(\boldsymbol{n}_1,\ldots,\boldsymbol{n}_m)\in \N^{\sum_{j=1}^d(w_{j}+1)}$ Pad\'{e}-type approximants of $(f_{j,u_{j}}(z))_{\substack{ 1\le j \le d \\ 0\le u_{j}\le w_{j}}}$, where $\boldsymbol{n}_j=(n_j,\ldots,n_j)\in \N^{w_j+1}$ for $1\le j \le d$.
\end{theorem}
\begin{proof}
By Lemma $\ref{equivalence Pade}$, it suffices to prove that any triple $(j,u_{j},k)$ with $1\le j \le d$, $0\le u_{j} \le w_{j}$, $0\le k \le n_j-1$ satisfy $t^kP(t)\in {\rm{ker}}\,\varphi_{f_{j,u_{j}}}$. 
Put $\mathcal{R}_{j,n_j}=\mathcal{R}_{D_{j},n_j,\vec{r}_{j}}$. Then we have $P(t)={\prod_{j=1}^d}\mathcal{R}_{j,n_j}\cdot F(t)$ and thus 
\begin{align} \label{decomposition i,j,k}
t^kP(t)=t^k\mathcal{R}_{j,n_j}\prod_{j^{\prime}\neq j}\mathcal{R}_{j^{\prime},n_{j^{\prime}}}\cdot F(t).
\end{align}
Since $F(t)\in (\prod_{v=1}^{l}a_{v}(z)^{r_{j,v}+\sum_{j'\neq j}^d r_{j',v}})$, using Corollary $\ref{im cor} \ (ii)$, we obtain 
$${\displaystyle\prod_{j^{\prime}\neq j}}\mathcal{R}_{j^{\prime},n_{j^{\prime}}}\cdot F(t)\in \biggl(\prod_{v=1}^{l}a_{v}(t)^{r_{j,v}}\biggr).$$
Combining Equation $(\ref{decomposition i,j,k})$ and the above relation yields
\begin{align*} 
t^kP(t)\in t^k\mathcal{R}_{j,n_j}\cdot  \biggl(\prod_{v=1}^{l}a_{v}(t)^{r_{j,v}}\biggr)\subseteq t^k\mathcal{E}^{n_j}_{a,b_{j}}\cdot (a(t)^{n_j})\subseteq {\rm{ker}}\,\varphi_{f_{j,u_{j}}}.
\end{align*}
Note that the last inclusion is obtained from Corollary $\ref{im cor} \ (i)$ for $D_{j}\cdot f_{j,u_{j}}(z)\in K[z]$.
\end{proof}

\subsection{Commutativity of differential operators}
In this subsection, we give a sufficient condition under which weighted Rodrigues operators commute.
We denote $\partial_z\cdot c(z)$ by $c^{\prime}(z)$ for any rational function $c(z)\in K(z)$.
\begin{lemma} \label{decompose}
Let $a(z),b(z)\in K[z]$ and $c(z)\in K(z)$ with $a(z)c(z)\neq 0$. 
Let $w(z)$ be a nonzero solution of $-a(z)\partial_z+b(z)$ in some differential extension $\mathcal{K}$ of $K(z)$ and $n$ a nonnegative integer. Put 
$$R_n=\dfrac{1}{n!}\left(\partial_z+\dfrac{b(z)}{a(z)}\right)^nc(z)^n\in K(z)[\partial_z].$$
Then, in the ring $\mathcal{K}[\partial_z]$, we have the equality: 
$$R_n=\dfrac{1}{n!}w(z)^{-1}\partial^n_z w(z)c(z)^n=\dfrac{1}{n!}R_{1}(R_{1}+c^{\prime}(z))\cdots (R_{1}+(n-1)c^{\prime}(z)).$$
\end{lemma}
\begin{proof}
The first equality is readily obtained using the identity $$\partial_zw(z)=w(z)\left(\partial_z+\dfrac{b(z)}{a(z)}\right).$$
The second equality is proved using the identity
\begin{align*}
\left(\partial_z+\dfrac{b(z)}{a(z)}\right){{c(z)^n}}&=\left[c(z)^{n-1}\left(\partial_z+\dfrac{b(z)}{a(z)}\right)+(n-1)c^{\prime}(z)c(z)^{n-2}\right]c(z)\\
&=c(z)^{n-1}(R_1+(n-1)c^{\prime}(z)).
\end{align*}
This completes the proof of Lemma \ref{decompose}.
\end{proof}

\begin{lemma} \label{commute 0}
Let $a(z),b_1(z),b_2(z),c(z)\in K[z]$ with $a(z)c(z)\neq 0$.
For a nonnegative integer $n$ and $j=1,2$, $$R_{j,n}=\dfrac{1}{n!}\left(\partial_z+\dfrac{b_j(z)}{a(z)}\right)^nc(z)^n.$$ 
Assume ${\rm{deg}}\,c\le 1$. Then the following are equivalent.

$(i)$ For any $n_1,n_2\in \N$, we have $R_{1,n_1}R_{2,n_2}=R_{2,n_2}R_{1,n_1}$.

$(ii)$ We have
$(b_2(z)-b_1(z))/a(z)c(z)\in K$.
\end{lemma}
\begin{proof}
Since ${\rm{deg}}\,c \le 1$ and therefore $c^{\prime}(z)\in K$, using Lemma $\ref{decompose}$, we see that item $(i)$ is equivalent to $R_{1,1}R_{2,1}=R_{2,1}R_{1,1}$.
Let us show that the commutativity of $R_{j,1}$ $(j=1,2)$ is equivalent to item $(ii)$.
According to the identity,
\begin{align*}
R_{1,1}R_{2,1}=R_{2,1}R_{1,1}+(R_{2,1}-R_{1,1})c^{\prime}(z)+\left(\dfrac{b_2(z)-b_1(z)}{a(z)}\right)'c(z)^{{2}},
\end{align*}
the identity $R_{1,1}R_{2,1}=R_{2,1}R_{1,1}$ is equivalent to 
\begin{align*}
(R_{2,1}-R_{1,1})c^{\prime}(z)+\left(\dfrac{b_2(z)-b_1(z)}{a(z)}\right)'c(z)^{{2}}&=\left(\dfrac{b_2(z)-b_1(z)}{a(z)}c^{\prime}(z)+\left(\dfrac{b_2(z)-b_1(z)}{a(z)}\right)^{\prime}c(z)\right){{c(z)}}\\
&=\left(\dfrac{b_2(z)-b_1(z)}{a(z)}c(z)\right)^{\prime}{{c(z)}}=0,
\end{align*}
which means item $(ii)$ holds. This completes the proof of Lemma $\ref{commute 0}$.
\end{proof}

\section{Determinants associated with Pad\'{e} approximants}

Let $f_{j,u_j}(z)$ be the Laurent series in Theorem $\ref{main general}$.
To consider the linear independence results on the values of $f_{j,u_j}(z)$ \'{a} la the method of Siegel (see \cite{Siegel}), we need to study nonvanishing of determinants of certain matrices.
In this section, we compute the determinants of specific matrices whose entries are given by the Pad\'{e} approximants of $f_{j,u_j}(z)$ obtained in {{Theorem}} $\ref{main general}$.

First, $d$ be a nonnegative integer and $a_1(z),a_2(z),b_1(z),\ldots,b_d(z)\in K[z]$. 
Put $a(z)=a_1(z)a_2(z)$, $w_j=\max\{{\rm{deg}}\,a-2,{\rm{deg}}\,b_j-1\}$ and $W=w_1+\cdots+w_d+d$.

Assume $w_j\ge0$, ${\rm{deg}}\,a_1\le 1$, $a_1$ is a monic polynomial and $$\gamma_{j_1,j_2}=\dfrac{b_{j_1}(z)-b_{j_2}(z)}{a_2(z)}\in K\setminus\{0\} \ \ \text{for} \ \ 1\le j_1<j_2\le d.$$ 

Denote $D_j=-a(z)\partial_z+b_j(z)\in K[z,\partial_z]$ {{and assume \eqref{assump a,b} for $D_j$.}}
{{Lemma $\ref{solutions}$ implies that there exist Laurent series $f_{j,0}(z),\ldots,f_{j,w_j}(z)\in (1/z)\cdot K[[1/z]]$ that are linearly independent over $K$ and satisfy 
$$D_j\cdot f_{j,u_j}(z)\in K[z] \ \  \text{for} \ \  1\le j \le d, \ \ 0\le u_j \le w_j.$$ 
We now fix these series.}}
For $n\in \N$, we denote the weighted Rodrigues operator associated with $D_j$ by
$$R_{j,n}=\dfrac{1}{n!}\left(\partial_z+\dfrac{b_j(z)}{a(z)}\right)^na_1(z)^n \ \  \text{for} \ \  1\le j \le d.$$
Lemma $\ref{commute 0}$ to the case of $a(z)=a_1(z)a_2(z)$ and $c(z)=a_1(z)$ asserts that the commutativity of the differential operators $R_{j,n}$, namely
$$R_{j_1,n}R_{j_2,n}=R_{j_2,n}R_{j_1,n} \ \ \text{for} \ \ 1\le j_1,j_2 \le d.$$
Put $\varphi_{j,u_j}=\varphi_{f_{j,u_j}}$. For $0\le h \le W$, we define 
\begin{align*}
&P_{n,h}(z)=P_h(z)=\prod_{j=1}^dR_{j,n}\cdot [z^ha_2(z)^{dn}],\\
&Q_{n,j,u_j,h}(z)=Q_{j,u_j,h}(z)=\varphi_{j,u_j}\left(\dfrac{P_{h}(z)-P_{h}(t)}{z-t} \right) \ \ \text{for} \ \ 1\le j \le d, \ 0\le u_j \le w_j,\\
&\mathfrak{R}_{n,j,u_j,h}(z)=\mathfrak{R}_{j,u_j,h}(z)=P_h(z)f_{j,u_j}(z)-Q_{j,u_j,h}(z) \ \ \text{for} \ \ 1\le j \le d, \ 0\le u_j \le w_j.
\end{align*}
Assume $P_h(z)\neq 0$. Theorem $\ref{main general}$ yields that the vector of polynomials $(P_h,Q_{j,u_j,h})_{\substack{1\le j \le d \\ {{0}} \le u_j \le w_j}}$ is a weight $(n,\ldots,n)\in \N^{W}$ Pad\'{e}-type approximants of $(f_{j,u_j})_{\substack{1\le j \le d \\ 0\le u_j \le w_j}}$.

First we compute the coefficients of $1/z^{n+1}$ of $\mathfrak{R}_{j,u_j,h}(z)$.
\begin{lemma} \label{coeff n+1}
Let notations be as above. 
For $1\le j \le d$, $0\le u_j \le w_j$ and $0\le h \le W$, we have 
\begin{align*} 
\mathfrak{R}_{j,u_j,h}(z)=\sum_{k=n}^{\infty}\dfrac{\varphi_{j,u_j}(t^kP_h(t))}{z^{k+1}}
\end{align*}
and 
$$\varphi_{j,u_j}(t^nP_h(t))=\dfrac{(-1)^n}{(n!)^{d-1}}\prod_{\substack{1\le j'\le d \\ j'\neq j}}\left[\prod_{k=1}^n(\gamma_{j',j}-k\varepsilon_{a_1})\right]\varphi_{j,u_j}(t^ha_1(t)^n\cdot {a}_2(t)^{dn}),$$
where $\varepsilon_{a_1}=1$ if ${\rm{deg}}\,a_1=1$ and $\varepsilon_{a_1}=0$ if ${\rm{deg}}\,a_1=0$.
\end{lemma}
\begin{proof}
Since $(\mathfrak{R}_{j,u_j,h})_{j,u_j}$ is a weight $(n,\ldots,n)\in \N^{W}$ Pad\'{e}-type approximation of $(f_{j,u_j})_{j,u_j}$, we have ${\rm{ord}}_{\infty}\,\mathfrak{R}_{j,u_j,h}\ge n+1$ and 
the first equality is obtained by
\begin{align*}
\mathfrak{R}_{j,u_j,h}(z)&=\varphi_{j,u_j}\left(\dfrac{P_h(t)}{z-t}\right)
=\sum_{k=n}^{\infty}\dfrac{\varphi_{j,u_j}(t^kP_h(t))}{z^{k+1}}.
\end{align*}
We prove the second equality. Fix $j$ and put $\mathcal{E}_{a,b_{j'}}=\partial_t+b_{j'}(t)/a(t)$ for $1\le j'\le d$. 
Then,
\begin{align} \label{j' j}
\mathcal{E}_{a,b_{j'}}=\mathcal{E}_{a,b_j}+\dfrac{\gamma_{j',j}}{a_1(t)},
\end{align}
and
$\mathcal{R}_{j',n}=(1/n!)\mathcal{E}^n_{a,b_{j'}}a_1(t)^n$.
Using Proposition $\ref{key 1 2}$ $(i)$, there is a set $\{c_{j,l}\mid l=0,1,\ldots ,n\}$ of integers with $c_{j,n}=(-1)^nn!$ and $$t^n\mathcal{R}_{j,n}=\sum_{l=0}^n\dfrac{c_{j,l}}{n!}\mathcal{E}^{n-l}_{a,b_j}t^{n-l}a_1(t)^n.$$
Note, by the Leibniz formula, the polynomial $\prod_{j'\neq j}\mathcal{R}_{j',n}\cdot [t^h{a}_2(t)^{dn}]$ is contained in the ideal $({a}_2(t)^n)$.
By Corollary $\ref{im cor}$ $(i)$, $$\mathcal{E}^{n-l}_{a,b_j} a_1(t)^n \cdot ({a}_2(t)^n)\subseteq {\rm{ker}}\,\varphi_{j,u_j} \ \ \text{for} \ \ 0\le l\le n-1$$
and thus,
\begin{align}
t^nP_h(t)&=t^n \mathcal{R}_{j,n}\prod_{j'\neq j}\mathcal{R}_{j',n}\cdot [t^h{a}_2(t)^{dn}] \nonumber \\
&=\sum_{l=0}^n\dfrac{c_{j,l}}{n!}\mathcal{E}^{n-l}_{a,b_j}t^{n-l}a_1(t)^n\prod_{j'\neq j}\mathcal{R}_{j',n}\cdot [t^ha_2(t)^{dn}]\nonumber \\
&\equiv (-1)^n a_1(t)^n\prod_{j'\neq j}\mathcal{R}_{j',n}\cdot [t^ha_2(t)^{dn}] \ \ \text{mod} \ \  {\rm{ker}}\,\varphi_{j,u_j}. \label{t^k 1}
\end{align}
Equation $(\ref{j' j})$ yields 
\begin{align}
a_1(t)^n\mathcal{R}_{j',n}&=\dfrac{a_1(t)^n}{n!}\left(\mathcal{E}_{a,b_j}+\dfrac{\gamma_{j',j}}{a_1(t)}\right)^na_1(t)^n \nonumber \\
&=\dfrac{1}{n!}\left(\mathcal{E}_{a,b_j}a_1(t)^n+(\gamma_{j',j}-n\varepsilon_{a_1})a_1(t)^{n-1}\right)\left(\mathcal{E}_{a,b_j}+\dfrac{\gamma_{j',j}}{a_1(t)}\right)^{n-1}a_1(t)^n \label{second!}\\
&\equiv\dfrac{1}{n!}(\gamma_{j',j}-n\varepsilon_{a_1})a_1(t)^{n-1}\left(\mathcal{E}_{a,b_j}+\dfrac{\gamma_{j',j}}{a_1(t)}\right)^{n-1}a_1(t)^n \ \ \text{mod} \ \ \{\mathcal{E}_{a,b_{j}}a_1(t)\cdot K[t,\partial_t]\} \nonumber\\
&\equiv \dfrac{1}{n!}\prod_{k=1}^n(\gamma_{j',j}-k\varepsilon_{a_1})a_1(t)^n \ \ \text{mod} \ \ \{\mathcal{E}_{a,b_{j}}a_1(t)\cdot K[t,\partial_t]\}. \nonumber
\end{align}
Note that we use the assumption ${\rm{deg}}\, a_1\le 1$ {{and the equality $\partial_t\cdot a(t)=\varepsilon_{a_1}$}} in Equation \eqref{second!}. 
Combining above equality and Equation $(\ref{t^k 1})$ yields 
$$\varphi_{j,u_j}(t^nP_h(t))=\dfrac{(-1)^n}{(n!)^{d-1}}\prod_{\substack{1\le j'\le d \\ j'\neq j}}\left[\prod_{k=1}^n(\gamma_{j',j}-k\varepsilon_{a_1})\right]\varphi_{j,u_j}(t^ha_1(t)^n\cdot a_2(t)^{dn}).$$
This completes the proof of Lemma $\ref{coeff n+1}$. 
\end{proof}
For a nonnegative integer $n$, we now consider the determinant of following $(W+1)\times (W+1)$ matrix:
\begin{align*}
&\Delta_n(z)=\Delta(z)={\rm{det}}
\begin{pmatrix}
P_{0}(z) & P_1(z) & \ldots & P_{W}(z)\\
Q_{1,0,0}(z) & Q_{1,0,1}(z) & \ldots & Q_{1,0,W}(z)\\
\vdots & \vdots & \ddots & \vdots\\
Q_{1,w_1,0}(z) & Q_{1,w_1,1}(z) & \ldots & Q_{1,w_1,W}(z)\\
\vdots & \vdots & \ddots & \vdots\\
Q_{d,0,0}(z) & Q_{d,0,1}(z) & \ldots & Q_{d,0,W}(z)\\
\vdots & \vdots & \ddots & \vdots\\
Q_{d,w_d,0}(z) & Q_{d,w_d,1}(z) & \ldots & Q_{d,w_d,W}(z)
\end{pmatrix}.
\end{align*}
Notice that the determinant $\Delta(z)$ is a polynomial.

To compute $\Delta(z)$, we define the determinant of following $W\times W$ matrix:
\begin{align*}
&\Theta_n=\Theta={\rm{det}}
\begin{pmatrix}
\varphi_{1,0}(a_1(t)^na_2(t)^{dn}) & \varphi_{1,0}(ta_1(t)^na_2(t)^{dn}) & \ldots & \varphi_{1,0}(t^{W-1}a_1(t)^na_2(t)^{dn})\\
\vdots & \vdots & \ddots & \vdots\\
\varphi_{1,w_1}(a_1(t)^na_2(t)^{dn}) & \varphi_{1,w_1}(ta_1(t)^na_2(t)^{dn}) & \ldots & \varphi_{1,w_1}(t^{W-1}a_1(t)^na_2(t)^{dn})\\
\vdots & \vdots & \ddots & \vdots\\
\varphi_{d,0}(a_1(t)^na_2(t)^{dn}) & \varphi_{d,0}(ta_1(t)^na_2(t)^{dn}) & \ldots & \varphi_{d,0}(t^{W-1}a_1(t)^na_2(t)^{dn})\\
\vdots & \vdots & \ddots & \vdots\\
\varphi_{d,w_d}(a_1(t)^na_2(t)^{dn}) & \varphi_{d,w_d}(ta_1(t)^na_2(t)^{dn}) & \ldots & \varphi_{d,w_d}(t^{W-1}a_1(t)^na_2(t)^{dn})
\end{pmatrix}.
\end{align*}
{{Notice that $\Theta\in K$.}}
{{Replace the coefficient of $z^{(n+1)W}$ of the polynomial $P_W$ by $p_{W}$, that is, $$p_W=\dfrac{1}{[(n+1)W]!}\partial^{(n+1)W}_z\cdot P_{W}(z).$$}}
Then we have the following proposition.
\begin{proposition} \label{det}
{{$\Delta(z)\in K$. More precisely,}}
$$\Delta(z)=\left(\dfrac{{{-1}}}{(n!)^{d-1}}\right)^W p_W\cdot \prod_{j=1}^d\left[\prod_{\substack{1\le j' \le d \\ j'\neq j}}\prod_{k=1}^n(\gamma_{j',j}-k\varepsilon_{a_1})\right]^{{{w_j+1}}}\cdot \Theta,$$
where $\varepsilon_{a_1}$ is the real number defined in Lemma $\ref{coeff n+1}$.
\end{proposition}
\begin{proof}
First, by the definition of $P_{l}(z)$, we have 
\begin{align} \label{deg Pl}
{\rm{deg}}\,P_{l}\le nW+l.
\end{align}
For the matrix in the definition of $\Delta(z)$, for $1\le j \le d, 0\le u_j\le w_j$, adding $-f_{j,u_j}(z)$ times first row to $(w_1+\cdots+w_{j-1})+u_j+1$ th row, 
                     $$ 
                     \Delta(z)=(-1)^{W}{\rm{det}}
                     {\begin{pmatrix}
                     P_{0}(z) & \dots &P_{W}(z)\\
                     \mathfrak{R}_{1,0,0}(z) & \dots & \mathfrak{R}_{1,0,W}(z)\\
                     \vdots    & \ddots & \vdots  \\
                     \mathfrak{R}_{1,w_1,0}(z) & \dots & \mathfrak{R}_{1,w_1,W}(z)\\
                     \vdots    & \ddots & \vdots  \\
                     \mathfrak{R}_{d,0,0}(z) & \dots & \mathfrak{R}_{d,0,W}(z)\\
                     \vdots    & \ddots & \vdots  \\
                     \mathfrak{R}_{d,w_d,0}(z) & \dots & \mathfrak{R}_{d,w_d,W}(z)\\
                     \end{pmatrix}}. 
                     $$ 
We denote the $(s,t)$ th cofactor of the matrix in the right hand side of above equality by $\Delta_{s,t}(z)$.
Then we have, developing along the first row 
\begin{align} \label{formal power series rep delta}
\Delta(z)=(-1)^{W}\left(\sum_{l=0}^{{{W}}}P_{l}(z)\Delta_{1,l+1}(z)\right).
\end{align} 
Since 
$${\rm{ord}}_{\infty}\, {{\mathfrak{R}}}_{l,h}(z)\ge n+1\ \ \text{for} \ \ 1\le j \le d, 0\le u_j \le w_j, 0\le h \le W,$$
$$
{\rm{ord}}_{\infty}\,\Delta_{1,l+1}(z)\ge (n+1)W \ \ \text{for} \ \ 0\le l \le W.
$$
Combining Equation $(\ref{deg Pl})$ and above inequality yields
$$P_{l}(z)\Delta_{1,l+1}(z)\in (1/z)\cdot K[[1/z]] \ \ \text{for} \ \ 0\le l \le W-1,$$
and 
$$P_{W}(z)\Delta_{1,{{W+1}}}(z)\in K[[1/z]].$$
Note that in above relation, the constant term of $P_{W}(z)\Delta_{1,{{W+1}}}(z)$ is 
\begin{equation} \label{what const}
{{p_W}} \cdot ``\text{Coefficient of} \ 1/z^{(n+1)W} \ \text{of} \ \Delta_{1,{{W+1}}}(z)\text{''}.
\end{equation}  
Equation $(\ref{formal power series rep delta})$ implies $\Delta(z)$ is a polynomial in $z$ with nonpositive valuation with respect to ${\rm{ord}}_{\infty}$. Thus, it has to be a constant.
At last, by Lemma $\ref{coeff n+1}$, the coefficient of $1/z^{(n+1)W}$ of  $\Delta_{1,{{W+1}}}(z)$ is 
{\small{\begin{align*}
{\rm{det}}
\begin{pmatrix}
(-1)^n\varphi_{1,0}(t^nP_0(t)) & \ldots & (-1)^n\varphi_{1,0}(t^{n}P_{W-1}(t)) \\
\vdots & \ddots & \vdots\\
(-1)^n\varphi_{1,w_1}(t^nP_0(t)) & \ldots & (-1)^n\varphi_{1,w_1}(t^{n}P_{W-1}(t)) \\
\vdots & \ddots & \vdots\\
(-1)^n\varphi_{d,0}(t^nP_0(t)) & \ldots & (-1)^n\varphi_{d,0}(t^{n}P_{W-1}(t)) \\
\vdots & \ddots & \vdots\\
(-1)^n\varphi_{d,w_d}(t^nP_0(t)) & \ldots & (-1)^n\varphi_{d,w_d}(t^{n}P_{W-1}(t))  
\end{pmatrix}
=\left(\dfrac{{{1}}}{(n!)^{d-1}}\right)^W\prod_{j=1}^d\left[\prod_{\substack{1\le j' \le d \\ j'\neq j}}\prod_{k=1}^n(\gamma_{j',j}-k\varepsilon_{a_1})\right]^{{w_j+1}}\cdot \Theta.
\end{align*}}}
Combining Equations. $(\ref{formal power series rep delta})$, $(\ref{what const})$ and above equality yields the assertion.
This completes the proof of Proposition $\ref{det}$.
\end{proof} 

\section{Examples}
In this section, let us describe some examples of the application of Theorem $\ref{main general}$ and Proposition $\ref{det}$.
\begin{example} \label{Chebyshev}
Let us give a generalization of the Chevyshev polynomials (see \cite[$5.1$]{Spe-func}).
Let $u\ge 2$ be an integer. Put $D=-(z^u-1)\partial_z-z^{u-1}\in K[z,\partial_z]$. The Laurent series
$$f_l(z)=\sum_{k=0}^{\infty}\dfrac{\left(\tfrac{1+l}{u}\right)_k}{\left(\tfrac{u+l}{u}\right)_k}\dfrac{1}{z^{uk+l+1}}=\dfrac{1}{z^{l+1}}\cdot {}_2F_1\left(\frac{1+l}{u},1,\left.\frac{u+l}{u}\right|\frac{1}{z^u}\right) \ \ \text{for} \ \ 0\le l \le u-2$$ 
are linearly independent over $K$ and satisfy $D\cdot f_l(z)\in K[z]$. Note that $f_0(z)=(z^u-1)^{-1/u}$. We denote $\varphi_{f_l}=\varphi_l$.
For $h,n\in \N$ with $0\le h \le u-1$, we define
\begin{align*}
&P_{n,h}(z)=P_h(z)=\dfrac{1}{n!} \left(\partial_z-\dfrac{z^{u-1}}{z^u-1}\right)^n(z^u-1)^n \cdot z^h,\\
&Q_{n,l,h}(z)=Q_{l,h}(z)=\varphi_{l}\left(\dfrac{P_{h}(z)-P_h(t)}{z-t}\right) \ \ \text{for} \ \ 0\le l \le u-2.
\end{align*}
Theorem $\ref{main general}$ yields that the vector of polynomials $(P_h,Q_{j,h})_{0\le j \le u-2}$ is a weight $(n,\ldots,n)\in \N^{u-1}$ Pad\'{e}-type approximants of $(f_0,\ldots,f_{u-2})$.
Define
$$\Delta_n(z)=
{\rm{det}}\begin{pmatrix}
P_0(z) & \cdots & P_{u-1}(z)\\
Q_{0,0}(z) & \cdots & Q_{0,u-1}(z)\\
\vdots & \ddots & \vdots \\
Q_{u-2,0}(z) & \cdots & Q_{u-2,u-1}(z)
\end{pmatrix}.
$$
The determinant $\Delta_n(z)$ will be computed in Lemma $\ref{non vanish uN}$.
\end{example}

\begin{example} \label{1}
In this example, we give a generalization of the Bessel polynomials (see \cite{Bessel}).
Let $d,n$ be nonnegative integers and $\gamma_1,\ldots,\gamma_d\in K$ that are not integers less than $-1$ with $$\gamma_{j_2}-\gamma_{j_1}\notin \Z \ \ \text{for} \ \ 1\le j_1<j_2\le d.$$ 
Put $D_j=-z^2\partial_z+\gamma_j z-1$, 
$$f_j(z)=\sum_{k=0}^{\infty}\dfrac{1}{(2+\gamma_j)_k}\dfrac{1}{z^{k+1}}$$ and $\varphi_{f_j}=\varphi_j$.
A straightforward computation yields $D_j\cdot f_j(z)\in K$.
Put $$R_{j,n}=\dfrac{1}{n!}\left(\partial_z+\dfrac{\gamma_jz-1}{z^2}\right)^nz^n.$$  
Lemma $\ref{commute 0}$ yields $$R_{j_1,n_1}R_{j_2,n_2}=R_{j_2,n_2}R_{j_1,n_1} \ \ \text{for} \ \ 1\le j_1,j_2 \le d \ \ \text{and} \ \ n_{j_1},n_{j_2}\in \N.$$
For $h\in \Z$ with $0\le h \le d$, we define 
\begin{align*}
&P_{n,h}(z)=P_h(z)=\prod_{j=1}^dR_{j,n}\cdot z^{dn+h},\\
&Q_{n,j}(z)=Q_{j}(z)=\varphi_{j}\left(\dfrac{P_h(z)-P_h(t)}{z-t}\right) \ \ \text{for} \ \ 1\le j \le d.
\end{align*}
Then Theorem $\ref{main general}$ yields that the vector of polynomials
$(P_h,Q_{j,h})_{1\le j \le d}$ is a weight $(n,\ldots,n)\in \N^d$ Pad\'{e}-type approximants of $(f_1,\ldots,f_d)$.
By the definition of $P_d(z)$, we have 
\begin{align} \label{Pd 1}
P_d(z)=\dfrac{\prod_{j=1}^d {{(d(n+1)+\gamma_j+1)_n}}}{(n!)^d}z^{d(n+1)}+{(\rm{lower \ degree \ terms})}.
\end{align} 

Define
\begin{align*}
&\Delta_n(z)={\rm{det}}
\begin{pmatrix}
P_{0}(z) & P_1(z) & \ldots & P_{d}(z)\\
Q_{1,0}(z) & Q_{1,1}(z) & \ldots & Q_{1,d}(z)\\
\vdots & \vdots & \ddots & \vdots\\
Q_{d,0}(z) & Q_{d,1}(z) & \ldots & Q_{d,d}(z)\\
\end{pmatrix}, \ \ 
\Theta_n={\rm{det}}
\begin{pmatrix}
\varphi_{1}(t^{(d+1)n}) &  \ldots & \varphi_{1}(t^{(d+1)n+d-1})\\
\vdots & \ddots & \vdots\\
\varphi_{d}(t^{(d+1)n}) &  \ldots & \varphi_{d}(t^{(d+1)n+d-1})
\end{pmatrix}.
\end{align*}
{{Let us compute $\Theta_n$. By the definition of $\varphi_j$ and the properties of determinants, we have
\begin{align*}
\Theta_n&=
{\rm{det}}
\begin{pmatrix}
\tfrac{1}{(2+\gamma_1)_{(d+1)n}} &  \ldots & \tfrac{1}{(2+\gamma_1)_{(d+1)n+d-1}}\\
\vdots & \ddots & \vdots\\
\tfrac{1}{(2+\gamma_d)_{(d+1)n}} &  \ldots & \tfrac{1}{(2+\gamma_d)_{(d+1)n+d-1}}
\end{pmatrix}\\
&=\prod_{j=1}^d\dfrac{1}{(2+\gamma_j)_{(d+1)n+d-1}}\cdot
{\rm{det}}
\begin{pmatrix}
(2+\gamma_1+(d+1)n)_{d-1}&  \ldots & (2+\gamma_1+(d+1)n)_{0}\\
\vdots & \ddots & \vdots\\
(2+\gamma_d+(d+1)n)_{d-1} &  \ldots & (2+\gamma_d+(d+1)n)_{0}
\end{pmatrix}.
\end{align*}
Here, by using the properties of determinant again, 
{\footnotesize{\begin{align*}
{\rm{det}}
\begin{pmatrix}
(2+\gamma_1+(d+1)n)_{d-1}&  \ldots & (2+\gamma_1+(d+1)n)_{0}\\
\vdots & \ddots & \vdots\\
(2+\gamma_d+(d+1)n)_{d-1} &  \ldots & (2+\gamma_d+(d+1)n)_{0}
\end{pmatrix}
&=
(-1)^{\tfrac{(d-1)d}{2}} 
{\rm{det}}
\begin{pmatrix}
(2+\gamma_1+(d+1)n)_{0} &  \ldots & (2+\gamma_1+(d+1)n)_{d-1}\\
\vdots & \ddots & \vdots \\
(2+\gamma_d+(d+1)n)_{0} &  \ldots &(2+\gamma_d+(d+1)n)_{d-1}
\end{pmatrix}\\
&=(-1)^{\tfrac{(d-1)d}{2}} 
{\rm{det}}
\begin{pmatrix}
1&  \gamma_1 & \ldots &\gamma^{d-1}_1\\
\vdots & \ddots & \vdots\\
1&  \gamma_d & \ldots &\gamma^{d-1}_d
\end{pmatrix}.
\end{align*}}}
Since the last determinant is a Vandermonde determinant,}}
$$\Theta_n=\prod_{j=1}^d\dfrac{1}{(2+\gamma_j)_{(d+1)n+d-1}}\cdot {{(-1)^{\tfrac{(d-1)d}{2}}}} \prod_{1\le j_1<j_2\le d}(\gamma_{j_2}-\gamma_{j_1}).$$
Proposition $\ref{det}$ and Equation $(\ref{Pd 1})$ imply that
{\small{
\begin{align*}
\Delta_n(z)={{(-1)^{\tfrac{(d-1)d}{2}}}}\left(\dfrac{{{-1}}}{(n!)^{d}}\right)^d\cdot \prod_{j=1}^d\left[\prod_{\substack{1\le j' \le d \\ j'\neq j}}\prod_{k=1}^n(\gamma_{j'}-\gamma_j-k)\right]\cdot \prod_{j=1}^d\dfrac{{{(d(n+1)+\gamma_j+1)_n}}}{(2+\gamma_j)_{(d+1)n+d-1}}
\cdot  \prod_{1\le j_1<j_2\le d}(\gamma_{j_2}-\gamma_{j_1}).
\end{align*}
}}
Especially we have $\Delta_n(z)\in K\setminus\{0\}$.
\end{example}

\begin{example} \label{2}
In this example, we give a generalization of the Laguerre polynomials (see \cite[$6.2$]{Spe-func}).
Let $d,n\in \N$, $\gamma_1,\ldots,\gamma_d\in K\setminus\{0\}$ be pairwise distinct and $\delta\in K$ be a nonnegative integer. 
Put $D_{j}=-z\partial_z-\gamma_jz+\delta$,
$$f_{j}(z)={\displaystyle\sum_{k=0}^{\infty}}(1+\delta)_k\left(\dfrac{1}{\gamma_jz}\right)^{k+1}$$
and $\varphi_{f_{j}}=\varphi_{j}$.
A straightforward computation shows $D_{j}\cdot f_{j}(z)\in K$.
Put $$R_{j,n}=\dfrac{1}{n!}\left(\partial_z-\dfrac{\gamma_j z-\delta}{z}\right)^{n}.$$ 
By Lemma $\ref{commute 0}$, we have $$R_{j_1,n_{j_1}}R_{j_2,n_{j_2}}=R_{j_2,n_{j_2}}R_{j_1,n_{j_1}} \ \ \text{for} \ \ 1\le j_1,j_2\le d, \ \ n_{j_1},n_{j_2}\in \N.$$
For $h\in \Z$ with $0\le h \le d$, we define 
\begin{align*}
&P_{n,h}(z)=P_h(z)=\prod_{j=1}^dR_{j,n}\cdot z^{dn+h},\\
&Q_{n,j}(z)=Q_{j}(z)=\varphi_{j}\left(\dfrac{P_h(z)-P_h(t)}{z-t}\right) \ \ \text{for} \ \ 1\le j \le d.
\end{align*}
Then Theorem $\ref{main general}$ yields that the vector of polynomials $(P_h,Q_{j,h})_{1\le j \le d}$ is a weight $(n,\ldots,n)\in \N^d$ Pad\'{e}-type approximants of $(f_{j})_{\substack{1\le j \le d}}$.
By the definition of $P_d(z)$, we have 
\begin{align} \label{Pd 2}
P_d(z)=\dfrac{\prod_{j=1}^d\gamma^n_j}{(n!)^d}z^{d(n+1)}+{(\rm{lower \ degree \ terms})}.
\end{align}

Define
\begin{align*}
&\Delta_n(z)={\rm{det}}
\begin{pmatrix}
P_{0}(z) & P_1(z) & \ldots & P_{d}(z)\\
Q_{1,0}(z) & Q_{1,1}(z) & \ldots & Q_{1,d}(z)\\
\vdots & \vdots & \ddots & \vdots\\
Q_{d,0}(z) & Q_{d,1}(z) & \ldots & Q_{d,d}(z)\\
\end{pmatrix}, \ \
\Theta_n={\rm{det}}
\begin{pmatrix}
\varphi_{1}(t^{dn}) &  \ldots & \varphi_{1}(t^{d(n+1)-1})\\
\vdots & \ddots & \vdots\\
\varphi_{d}(t^{dn}) &  \ldots & \varphi_{d}(t^{d(n+1)-1})
\end{pmatrix}.
\end{align*}
{{We now compute $\Theta_n$. 
By the definition of $\varphi_j$ and the properties of determinant, 
\begin{align*}
\Theta_n=
{\rm{det}}
\begin{pmatrix}
\tfrac{(1+\delta)_{dn}}{\gamma^{dn+1}_1} &  \ldots & \tfrac{(1+\delta)_{d(n+1)-1}}{\gamma^{d(n+1)}_1}\\
\vdots  & \ddots & \vdots\\
\tfrac{(1+\delta)_{dn}}{\gamma^{dn+1}_d} &  \ldots & \tfrac{(1+\delta)_{d(n+1)-1}}{\gamma^{d(n+1)}_d}
\end{pmatrix}
=\prod_{j=1}^d\dfrac{(1+\delta)_{dn+j-1}}{\gamma^{d(n+1)}_j}\cdot
{\rm{det}} \begin{pmatrix}
1 &  \gamma_1 & \ldots & \gamma^{d-1}_1\\
\vdots & \vdots & \ddots & \vdots\\
1 &  \gamma_d & \ldots & \gamma^{d-1}_d\\
\end{pmatrix}.
\end{align*}
Since the last determinant is nothing but a Vandermonde determinant,}} $$\Theta_n=\prod_{j=1}^d\dfrac{(1+\delta)_{dn+j-1}}{\gamma^{d(n+1)}_j}\cdot \prod_{1\le j_1<j_2\le d}(\gamma_{j_2}-\gamma_{j_1}).$$
Proposition $\ref{det}$ and Equation $(\ref{Pd 2})$ imply that
\begin{align*}
\Delta_n(z)=\left(\dfrac{{{-1}}}{(n!)^{d}}\right)^d\cdot \prod_{j=1}^d\left[\prod_{\substack{1\le j' \le d \\ j'\neq j}}(\gamma_{j'}-\gamma_j)^n\right]\cdot
\prod_{j=1}^d\dfrac{(1+\delta)_{dn+j-1}}{\gamma^{(d-1)n+d}_j}\cdot \prod_{1\le j_1<j_2\le d}(\gamma_{j_2}-\gamma_{j_1})\in K\setminus \{0\}.
\end{align*}
\end{example}

\begin{example} \label{3}
Let us give an alternative generalization of the Laguerre polynomials.
Let $d,n\in \N$, $\gamma\in  K\setminus\{0\}$, and $\delta_1,\ldots,\delta_d\in K$ be nonnegative integers with 
$$\delta_{j_1}-\delta_{j_2}\notin \Z \ \ \text{for} \ \ 1\le j_1<j_2\le d.$$ 
Put $D_{j}=-z\partial_z-\gamma z+\delta_j$,
$$f_{j}(z)={\displaystyle\sum_{k=0}^{\infty}}(1+\delta_j)_k\left(\dfrac{1}{\gamma z}\right)^{k+1},$$
and $\varphi_{f_{j}}=\varphi_{j}$.
Then we have $D_{j}\cdot f_{j}(z)\in K$.
Put $$R_{j,n}=\dfrac{1}{n!}\left(\partial_z-\dfrac{\gamma z-\delta_j}{z}\right)^{n}z^n.$$ 
By Lemma $\ref{commute 0}$, we have $$R_{j_1,n_{j_1}}R_{j_2,n_{j_2}}=R_{j_2,n_{j_2}}R_{j_1,n_{j_1}} \ \ \text{for} \ \ 1\le j_1,j_2\le d, \ \ n_{j_1},n_{j_2}\in \N.$$
For $h\in \Z$ with $0\le h \le d$, we define 
\begin{align*}
&P_{n,h}(z)=P_h(z)=\prod_{j=1}^dR_{j,n}\cdot z^h,\\
&Q_{n,j}(z)=Q_{j}(z)=\varphi_{j}\left(\dfrac{P_h(z)-P_h(t)}{z-t}\right) \ \ \text{for} \ \ 1\le j \le d.
\end{align*}
Then Theorem $\ref{main general}$ yields that the vector of polynomials $(P_h,Q_{j,h})_{1\le j \le d}$ is a weight $(n,\ldots,n)\in \N^d$ Pad\'{e}-type approximants of $(f_{j})_{\substack{1\le j \le d}}$.
By the definition of $P_d(z)$, 
\begin{align} \label{Pd 3}
P_d(z)=\dfrac{\gamma^{dn}}{(n!)^d}z^{d(n+1)}+{(\rm{lower \ degree \ terms})}.
\end{align}

Define
\begin{align*}
&\Delta_n(z)={\rm{det}}
\begin{pmatrix}
P_{0}(z) & P_1(z) & \ldots & P_{d}(z)\\
Q_{1,0}(z) & Q_{1,1}(z) & \ldots & Q_{1,d}(z)\\
\vdots & \vdots & \ddots & \vdots\\
Q_{d,0}(z) & Q_{d,1}(z) & \ldots & Q_{d,d}(z)\\
\end{pmatrix}, \ \
\Theta_n={\rm{det}}
\begin{pmatrix}
\varphi_{1}(t^{n}) &  \ldots & \varphi_{1}(t^{d+n-1})\\
\vdots & \ddots & \vdots\\
\varphi_{d}(t^{n}) &  \ldots & \varphi_{d}(t^{d+n-1})
\end{pmatrix}.
\end{align*}
{{Let us compute $\Theta_n$.
By the definition of $\varphi_j$ and the properties of determinant, 
\begin{align*}
\Theta_n=
{\rm{det}}
\begin{pmatrix}
\tfrac{(1+\delta_1)_{n}}{\gamma^{n+1}} &  \ldots & \tfrac{(1+\delta_1)_{d+n-1}}{\gamma^{d+n}}\\
\vdots & \ddots & \vdots\\
\tfrac{(1+\delta_d)_{n}}{\gamma^{n+1}} &  \ldots & \tfrac{(1+\delta_d)_{d+n-1}}{\gamma^{d+n}}
\end{pmatrix}
=\prod_{j=1}^d\dfrac{(1+\delta_j)_{n}}{\gamma^{n+j}}
\cdot 
\begin{pmatrix}
(n+\delta_1)_0 &  \ldots & (n+\delta_1)_{d-1}\\
\vdots & \ddots & \vdots\\
(n+\delta_1)_0 &  \ldots & (n+\delta_d)_{d-1}
\end{pmatrix}.
\end{align*}
A similar computation as in Example $\ref{1}$ leads us to get}} $$\Theta_n=\prod_{j=1}^d\dfrac{(1+\delta_j)_{n}}{\gamma^{n+j}}\cdot \prod_{1\le j_1<j_2\le d}(\delta_{j_2}-\delta_{j_1}).$$
Proposition $\ref{det}$ and Equation $(\ref{Pd 3})$ imply that
\begin{align*}
\Delta_n(z)=\left(\dfrac{{{-1}}}{(n!)^{d}}\right)^d\cdot \prod_{j=1}^d\left[\prod_{\substack{1\le j' \le d \\ j'\neq j}}\prod_{k=1}^n(\delta_{j'}-\delta_j-k)\right]\cdot
\prod_{j=1}^d\dfrac{(1+\delta_j)_{n}}{\gamma^{j}}\cdot \prod_{1\le j_1<j_2\le d}(\delta_{j_2}-\delta_{j_1})
\in K\setminus \{0\}.
\end{align*}
\end{example}

\begin{example} \label{4}
In this example, we give a generalization of the Hermite polynomials  (see \cite[$6.1$]{Spe-func}). 
Let $d,n\in \N$, $\gamma\in K\setminus\{0\}$ and $\delta_1,\ldots,\delta_d\in K$ be pairwise distinct. Put $D_j=-\partial_z+\gamma z+\delta_j$,
$$f_j(z)={\displaystyle\sum_{k=0}^{\infty}}\dfrac{f_{j,k}}{z^{k+1}},$$ where $f_{j,0}=1$, $f_{j,1}=-\delta_j/\gamma$ and 
\begin{align} \label{recurrence}
f_{j,k+2}=-\dfrac{\delta_jf_{j,k+1}+(k+1)f_{j,k}}{\gamma} \ \ \text{for} \ \ k\ge 0,
\end{align}
and $\varphi_{f_j}=\varphi_j$. 
Then we have $D_j\cdot f_{j}(z)\in K$.
Put $$R_{j,n}=\dfrac{1}{n!}(\partial_z+\gamma z+\delta_j)^{n}.$$
By Lemma $\ref{commute 0}$,  $$R_{j_1,n_{j_1}}R_{j_2,n_{j_2}}=R_{j_2,n_{j_2}}R_{j_1,n_{j_1}} \ \ \text{for} \ \ 1\le j_1,j_2\le d, \ \ n_{j_1},n_{j_2}\in \N.$$
For $h\in \Z$ with $0\le h \le d$, we define
\begin{align*}
&P_{n,h}(z)=P_h(z)=\prod_{j=1}^dR_{j,n}\cdot z^h,\\
&Q_{n,j,h}(z)=Q_{j,h}(z)=\varphi_{j}\left(\dfrac{P_h(z)-P_h(t)}{z-t}\right) \ \ \text{for} \ \ 1\le j \le d.
\end{align*}
Then Theorem $\ref{main general}$ yields that the vector of polynomials $(P_h,Q_{j,h})_{\substack{1\le j \le d}}$ is a weight $(n,\ldots,n)\in \N^d$ Pad\'{e}-type approximants of $(f_{j})_{1\le j \le d}$.
By the definition of $P_d(z)$, 
\begin{align} \label{Pd 4}
P_d(z)=\dfrac{\gamma^{dn}}{(n!)^d}z^{d(n+1)}+{(\rm{lower \ degree \ terms})}.
\end{align}

Define
\begin{align*}
\Delta_n(z)={\rm{det}}
\begin{pmatrix}
P_{0}(z) & P_1(z) & \ldots & P_{d}(z)\\
Q_{1,0}(z) & Q_{1,1}(z) & \ldots & Q_{1,d}(z)\\
\vdots & \vdots & \ddots & \vdots\\
Q_{d,0}(z) & Q_{d,1}(z) & \ldots & Q_{d,d}(z)\\
\end{pmatrix}, \ \
\Theta_n={\rm{det}}
\begin{pmatrix}
\varphi_{1}(1) &  \ldots & \varphi_{1}(t^{d-1})\\
\vdots & \ddots & \vdots\\
\varphi_{d}(1) &  \ldots & \varphi_{d}(t^{d-1})
\end{pmatrix}.
\end{align*}
{{Let us compute $\Theta_n$. By the definition of $\varphi_j$, 
\begin{align}\label{by def Theta}
\Theta_n=
{\rm{det}}
\begin{pmatrix}
f_{1,0} &  f_{1,1}& \ldots & f_{1,d-1}\\
\vdots &\vdots & \ddots & \vdots\\
f_{d,0} & f_{d,1} & \ldots & f_{d,d-1}
\end{pmatrix}.
\end{align}
Here, using Equation $(\ref{recurrence})$ and the properties of determinant  repeatedly,
\begin{align}\label{comp f}
{\rm{det}}
\begin{pmatrix}
f_{1,0} &  f_{1,1}& \ldots & f_{1,d-1}\\ 
\vdots &\vdots & \ddots & \vdots\\
f_{d,0} & f_{d,1} & \ldots & f_{d,d-1}
\end{pmatrix}
=
{\rm{det}}
\begin{pmatrix}
1 &  \frac{-\delta_1}{\gamma}& \ldots & \left(\frac{-\delta_1}{\gamma}\right)^{d-1}\\
\vdots &\vdots & \ddots & \vdots\\
1 &  \frac{-\delta_d}{\gamma}& \ldots & \left(\frac{-\delta_d}{\gamma}\right)^{d-1}
\end{pmatrix}.
\end{align}
Combining Equations (\ref{by def Theta}) and (\ref{comp f}) implies}}
$$\Theta_n=\left(\dfrac{-1}{\gamma}\right)^{1+2+\cdots+(d-1)}\cdot \prod_{1\le j_1<j_2\le d}(\delta_{j_2}-\delta_{j_1}).$$
Proposition $\ref{det}$ and Equation $(\ref{Pd 4})$ imply that
{\small{
\begin{align*}
\Delta_n(z)=\left(\dfrac{{{-1}}}{(n!)^{d}}\right)^d\cdot \prod_{j=1}^d\left[\prod_{\substack{1\le j' \le d \\ j'\neq j}}(\delta_{j'}-\delta_j)^n\right]\cdot
(-1)^{\tfrac{(d-1)d}{2}}\gamma^{dn-\tfrac{(d-1)d}{2}}\cdot \prod_{1\le j_1<j_2\le d}(\delta_{j_2}-\delta_{j_1})
\in K\setminus \{0\}.
\end{align*}
}}
\end{example}

\begin{example} \label{6}
In this example, we consider a generalization of the Legendre polynomials  (see \cite[Remark $5.3.1$]{Spe-func}). 
Let $d,m,n\in \N$, $\alpha_{1},\ldots,\alpha_{m}\in K\setminus\{0\}$ be pairwise distinct and $\gamma_1,\ldots,\gamma_d\in K$ be nonnegative integers{{,}} satisfying $\gamma_{j_1}-\gamma_{j_2}\notin \Z$ for $1\le j_1<j_2\le d$. 
Put $a_2(z)=\prod_{i=1}^m(z-\alpha_i)$, $D_{j}=-za_2(z)\partial_z+\gamma_ja_2(z)$,
$$f_{i,j}(z)={\displaystyle\sum_{k=0}^{\infty}}\dfrac{1}{k+1+\gamma_j}\left(\dfrac{\alpha_{i}}{z}\right)^{k+1} \ \ \text{for} \ \ 1\le i \le m, \ \ 1\le j \le d,$$
and $\varphi_{f_{i,j}}=\varphi_{i,j}$.
Then we have $D_{j}\cdot f_{i,j}(z)\in K[z]$. 
Put $$R_{j,n}=\dfrac{1}{n!}\left(\partial_z+\dfrac{\gamma_j}{z}\right)^nz^n.$$
By Lemma $\ref{commute 0}$, we have $$R_{j_1,n_{j_1}}R_{j_2,n_{j_2}}=R_{j_2,n_{j_2}}R_{j_1,n_{j_1}} \ \ \text{for} \ \ 1\le j_1,j_2\le d, \ \ n_{j_1},n_{j_2}\in \N.$$
For $h\in \Z$ with $0\le h \le dm$, we define
\begin{align*}
&P_{n,h}(z)=P_h(z)=\prod_{j=1}^dR_{j,n}\cdot [z^ha_2(z)^{dn}],\\
&Q_{n,i,j,h}(z)=Q_{i,j,h}(z)=\varphi_{i,j}\left(\dfrac{P_h(z)-P_h(t)}{z-t}\right) \ \ \text{for} \ \  1\le i \le m, \ \ 1\le j \le d.
\end{align*}
Then Theorem $\ref{main general}$ yields that the vector of polynomials
$(P_h,Q_{i,j,h})_{\substack{1\le i \le m \\ 1\le j \le d }}$ is a weight $(n,\ldots,n)\in \N^{dm}$ Pad\'{e} approximants of $(f_{i,j})_{\substack{1\le i \le m \\ 1\le j \le d}}$.
Define
\begin{align*}
&\Delta_n(z)={\rm{det}}
\begin{pmatrix}
P_{0}(z) & P_1(z) & \ldots & P_{dm}(z)\\
Q_{1,1,0}(z) & Q_{1,1,1}(z) & \ldots & Q_{1,1,dm}(z)\\
\vdots & \vdots & \ddots & \vdots\\
Q_{m,1,0}(z) & Q_{m,1,1}(z) & \ldots & Q_{m,1,dm}(z)\\
\vdots & \vdots & \ddots & \vdots\\
Q_{1,d,0}(z) & Q_{1,d,1}(z) & \ldots & Q_{1,d,dm}(z)\\
\vdots & \vdots & \ddots & \vdots\\
Q_{m,d,0}(z) & Q_{m,d,1}(z) & \ldots & Q_{m,d,dm}(z)\\
\end{pmatrix}.
\end{align*}
The nonvanishing of $\Delta_n(z)$ has been proven in \cite[Proposition $4.1$]{DHK3}. 
\end{example}
\begin{remark} \label{pre results}
We mention that Examples $\ref{1}$, $\ref{2}$, $\ref{3}$ and $\ref{6}$ can be applicable to prove the linear independence of the values of the series which are considered in each example. However such results have been obtained as follows.
 
In Example $\ref{1}$, for $\gamma_1,\ldots,\gamma_d\in \Q$, the series $f_{j}(z)$ become $E$-functions in the sense of Siegel (see \cite{Siegel}).  
The linear independence result for the values of these $E$-functions has been studied by V$\ddot{\text{a}}$$\ddot{\text{a}}$n$\ddot{\text{a}}$nen in \cite{Va3}.
In Example $\ref{2}$, for $\delta\in \Q$ and $\gamma_1,\ldots,\gamma_d\in K$ for an algebraic number field $K$, the series $f_j(z)$ are Euler-type series.
In the case of $\delta=0$, the global relations among the values of these Euler-type series have been studied by Matala-aho and Zudilin for $d=1$ in \cite{Ma-Zu} and Sepp$\ddot{\text{a}}$l$\ddot{\text{a}}$ for general $d$ in \cite{Sep}.
Likewise, Example $\ref{3}$, for $\delta_1,\ldots,\delta_d \in \Q$ and $\gamma=1$, treats Euler-type series.
In \cite{Va1}, V$\ddot{\text{a}}$$\ddot{\text{a}}$n$\ddot{\text{a}}$nen studied the global relations among the values of these Euler-type series.
In Example $\ref{6}$, for $\gamma_1,\ldots,\gamma_d\in \Q$ and $\alpha_1,\ldots,\alpha_m\in K$ for an algebraic number field $K$, the series $f_{i,j}(z)$ become $G$-functions in the sense of Siegel (see \cite{Siegel}) called the first Lerch functions. The linear independence of values of these functions has been studied by David, Hirata-Kohno and the author in \cite[Theorem $2.1$]{DHK3}.
\end{remark}

\section{Proof of Theorem $\ref{main!}$} 
{{This section is devoted to the proof of Theorem $\ref{main!}$. We proved the more precise theorem that we will state below. To state the theorem, we prepare the notation.}} 

Let $K$ an algebraic number field. We denote the set of places of $K$ by ${{\mathfrak{M}}}_K$. 
For $v\in {{\mathfrak{M}}}_K$, we denote the completion of $K$ with respect to $v$ by $K_v$ and define the normalized absolute value $|\cdot |_v$ as follows:
\begin{align*}
&|p|_v=p^{-\tfrac{[K_v:\Q_p]}{[K:\Q]}} \ \text{if} \ \ v\mid p, \ \ \ \ \ \ |x|_v=|\iota_v x|^{\tfrac{[K_v:\R]}{[K:\Q]}} \ \text{if} \ \ v\mid \infty,
\end{align*}
where $p$ is a prime number and $\iota_v$ the embedding $K\hookrightarrow \C$ corresponding to $v$. 

Let $\beta\in K$, we define the absolute Weil height of  $\beta$ as 
\begin{align*}
&{{\mathrm{H}}({\beta})=\prod_{v\in {{\mathfrak{M}}}_K} \max\{ 1,|\beta|_v\},}
\end{align*}
Let $m$ be a positive integer and $\boldsymbol{\beta}=(\beta_0,{\ldots},\beta_m) \in\mathbb{P}_m(K)$.  
We define the absolute Weil height of $\boldsymbol{\beta}$ by
\begin{align*}
&{\mathrm{H}}(\boldsymbol{\beta})=\prod_{v\in {{\mathfrak{M}}}_K} \max\{ |\beta_0|_v,\ldots,|\beta_m|_v\}.
\end{align*}
{{and the logarithmic absolute Weil height by ${\mathrm{h}}(\boldsymbol{\beta})={\rm{log}}\, \mathrm{H}(\boldsymbol{\beta})$. 
Let $v\in \mathfrak{M}_K$, then ${\mathrm{h}}_v(\boldsymbol{\beta})=\log\Vert \boldsymbol{\beta}\Vert_v$ where $\Vert\cdot\Vert_v$ is the sup $v$-adic norm. Then we have ${\mathrm{h}}(\boldsymbol{\beta})={\displaystyle{\sum_{v\in \mathfrak{M}_K}}}{\mathrm{h}}_v(\boldsymbol{\beta})$ and for $\beta\in K$, ${\mathrm{h}}(\beta)$ is the height of the point $(1,\beta)\in\mathbb{P}_1(K)$. 

Let $u$ be an integer with $u\ge 2$. We put $\nu(u)=u\prod_{q:\text{prime}, q|u}q^{1/(q-1)}$.
Let $v_0$ be a place of $K$, $\alpha\in K$ with $|\alpha|_{v_0}>2$. 
In the case of $v_0$ is a nonarchimedean place, we denote the prime number under $v_0$ by $p_{v_0}$ and put 
$\varepsilon_{v_0}(u)=1$ if $u$ is coprime with $p_{v_0}$ and $\varepsilon_{v_0}(u)=0$ if $u$ is divisible by $p_{v_0}$.
We denote Euler's totient function by $\varphi$.

We define real numbers 
{\small{\begin{align*}
&\mathbb{A}_{v_0}(\alpha)={\rm{h}}_{v_0}(\alpha)-\begin{cases} {\rm{h}}_{v_0}(2) & \ \ \text{if} \ v_0\mid \infty \\ 
\dfrac{\varepsilon_{v_0}(u) \log\,|p_{v_0}|_{v_0}}{p_{v_0}-1} & \ \ \text{if} \ v_0\nmid \infty,
\end{cases}\\
&\mathbb{B}_{v_0}(\alpha)=(u-1){\rm{h}}(\alpha)+(u+1){\rm{h}}(2)+\dfrac{(2u-1)\log\,\nu(u)}{u}+\dfrac{u-1}{\varphi(u)}
-(u-1){\rm{h}}_{v_0}(\alpha)-\begin{cases} 
(u+1){\rm{h}}_{v_0}(2) & \ \ \text{if} \ v_0\mid \infty \\ 
\log\,|\nu(u)|^{-1}_{v_0} & \ \ \text{if} \ v_0\nmid \infty,
\end{cases}\\
&U_{v_0}(\alpha)=(u-1){\rm{h}}_{v_0}(\alpha)+
\begin{cases} (u+1){\rm{h}}_{v_0}(2)& \ \ \text{if} \ v_0\mid \infty \\ 
\log\,|\nu(u)|^{-1}_{v_0} & \ \ \text{if} \ v_0\nmid \infty,
\end{cases}\\
&V_{v_0}(\alpha)=\mathbb{A}_{v_0}(\alpha)-\mathbb{B}_{v_0}(\alpha).
\end{align*}}}
We can now state the following theorem.
\begin{theorem} \label{special hypergeometric}
Assume $V_{v_0}(\alpha)>0$. 
Then, for any positive number $\varepsilon$ with $\varepsilon<V_{v_0}(\alpha)$, there exists an effectively computable positive number $H_0$ depending on $\varepsilon$ and the given data such that the following property holds.
For any $\boldsymbol{\lambda}=(\lambda,\lambda_{l})_{0\le l \le u-2} \in K^{u} \setminus \{ \bold{0} \}$ satisfying $H_0\le {\mathrm{H}}(\boldsymbol{\lambda})$, then 
\begin{align*}
\left|\lambda+\sum_{l=0}^{u-2}\lambda_{l}\cdot \dfrac{1}{\alpha^{l+1}} {}_2F_1\left(\frac{1+l}{u},1, \left.\frac{u+l}{u}\,\right|\frac{1}{\alpha^u}\right)\right|_{v_0}>C(\alpha,\varepsilon){\mathrm{H}}_{v_0}(\boldsymbol{\lambda}) {\mathrm{H}}(\boldsymbol{\lambda})^{-\mu(\alpha,\varepsilon)},
\end{align*}
where 
\begin{align*}
\mu(\alpha,\varepsilon)=
\dfrac{\mathbb{A}_{v_0}(\alpha)+{{U}}_{v_0}(\alpha)}{V_{v_0}(\alpha)-\varepsilon}  \mbox{and}  C(\alpha,\varepsilon)=
\exp \left(-\left(\frac{\log(2)}{V_{v_0}(\alpha)-\varepsilon}+1\right)(\mathbb{A}_{v_0}(\alpha)+{{U}}_{v_0}(\alpha))\right).
\end{align*}
\end{theorem}
{{We derive Theorem $\ref{main!}$ from Theorem $\ref{special hypergeometric}$.
\begin{proof}[Proof of Theorem $\ref{main!}$]
Let us consider the case of $K=\Q$, $v_0=\infty$ and $\alpha\in \Z\setminus \{0,\pm 1\}$. Then we see that $V_{\infty}(\alpha)=V(\alpha)$ where $V(\alpha)$ is defined in Theorem $\ref{main!}$.
Assume $V(\alpha)>0$. Choose some $\boldsymbol{\lambda}=(\lambda,\lambda_{0}\ldots,\lambda_{u-2})\in \Q^{u}\setminus\{\bold{0}\}$ such that 
$$\lambda_0+\sum_{l=0}^{u-2}\lambda_{l}\cdot \dfrac{1}{\alpha^{l+1}} {}_2F_1\left(\frac{1+l}{u},1, \left.\frac{u+l}{u}\,\right|\frac{1}{\alpha^u}\right)=0.$$ 
If ${\rm{H}}(\boldsymbol{\lambda})\geq H_0$ (where $H_0$ is as in Theorem $\ref{special hypergeometric}$), there is nothing more to prove. Otherwise, let $m>0$ be a rational integer such that ${\rm{H}}(m\boldsymbol{\lambda})\geq H_0$. Then Theorem $\ref{special hypergeometric}$ ensures that
$$m\left(\lambda_0+\sum_{l=0}^{u-2}\lambda_{l}\cdot \dfrac{1}{\alpha^{l+1}} {}_2F_1\left(\frac{1+l}{u},1, \left.\frac{u+l}{u}\,\right|\frac{1}{\alpha^u}\right)\right)\neq 0.$$ 
This is a contradiction and completes the proof of Theorem $\ref{main!}$. 
\end{proof}
}}
{{Now we start the proof of Theorem $\ref{special hypergeometric}$.}}
The proof is relying on the Pad\'{e} approximants obtained in Example $\ref{Chebyshev}$.
In the following, we use the same notation as in Example $\ref{Chebyshev}$.
\subsection{Computation of determinants}
{{\begin{lemma} \label{non vanish uN}
Let $n$ be a positive integer. Put $n=uN+s$ for nonnegative integers $N,s$ with $0\le s \le u-1$. Then,
\begin{align*}
\Delta_{n}(z)=(-1)^{(uN+s+1)(u-1)}\dfrac{((uN+s+1)u-1-uN)_{uN+s}}{(uN+s)!}  \prod_{l=0}^{u-2}\dfrac{\left(\tfrac{u-1}{u}\right)_{uN+s}}{\left(\tfrac{u+l}{u}\right)_{uN+s}}\in K\setminus \{0\}.
\end{align*}
\end{lemma}}}
\begin{proof}
Put 
$$
\Theta_n=
{\rm{det}}\begin{pmatrix}
\varphi_{0}((t^u-1)^n) & \ldots & \varphi_0(t^{u-2}(t^u-1)^n)\\
\vdots & \ddots & \vdots\\
\varphi_{u-2}((t^u-1)^n) & \ldots & \varphi_{u-2}(t^{u-2}(t^u-1)^n)
\end{pmatrix}.
$$
Proposition $\ref{det}$ implies that 
$$\Delta_n(z)=(-1)^{(u-1)}\times \dfrac{1}{[(n+1)(u-1)]!}\partial^{(n+1)(u-1)}_z\cdot P_{u-1}(z)\times \Theta_n.$$
According to the definition of $P_{u-1}(z)$, 
\begin{align*}
\dfrac{1}{[(n+1)(u-1)]!}\partial^{(n+1)(u-1)}_z\cdot P_{u-1}(z)=
\dfrac{((n+1)u-1-n)_n}
{n!}.
\end{align*}
By the definition of $f_l$, 
$$
\varphi_l(t^k)=\begin{cases}
\dfrac{\left(\tfrac{1+l}{u}\right)_N}{\left(\tfrac{u+l}{u}\right)_N} & \ \ \text{if} \ \ k=uN+l \ \text{for some} \ N\in \Z,\\
0 & \ \ \text{otherwise}.
\end{cases}
$$
Above equality shows 
\begin{align}
\Theta_{n}&=
{\rm{det}}\begin{pmatrix}
\varphi_{0}((t^u-1)^{uN+s}) & 0 & \ldots & 0 \\
0& \varphi_{1}(t(t^u-1)^{uN+s})  & \ldots & 0 \\
\vdots & \vdots & \ddots & \vdots\\
0 & 0& \ldots & \varphi_{u-2}(t^{u-2}(t^u-1)^{uN+s})
\end{pmatrix} \nonumber\\
&=\prod_{l=0}^{u-2}\varphi_{l}(t^{l}(t^u-1)^{uN+s}). \label{Theta}
\end{align}
We now compute $\varphi_{l}(t^{l}(t^u-1)^{uN+s})$.
Since we have $$t^{l}(t^u-1)^{uN+s}=\sum_{v=0}^{uN+s}\binom{uN+s}{v}(-1)^{uN+l-v}t^{uv+l},$$
we obtain
\begin{align*}
\varphi_{l}(t^{l}(t^u-1)^{uN+s})=\sum_{v=0}^{uN+s}\binom{uN+s}{v}(-1)^{uN+l-v}\dfrac{\left(\tfrac{1+l}{u}\right)_v}{\left(\tfrac{u+l}{u}\right)_v}.
\end{align*}
For positive real numbers $\alpha,\beta$ with $\alpha<\beta$ and a nonnegative integer $v$, 
\begin{align*}
\dfrac{(\alpha)_v}{(\beta)_v}
&=\dfrac{\Gamma(\beta)}{\Gamma(\alpha)\Gamma(\beta-\alpha)}\int^1_0 \xi^{\alpha+v-1}(1-\xi)^{\beta-\alpha-1}d\xi . 
\end{align*}
Applying above equality for $\alpha=(1+l)/u,\beta=(u+l)/u$, we obtain
\begin{align}
\varphi_{l}(t^{l}(t^u-1)^{uN+s})&=\dfrac{\Gamma(\tfrac{u+l}{u})}{\Gamma(\tfrac{1+u}{u})\Gamma(\tfrac{u-1}{u})}\sum_{v=0}^{uN+s}\binom{uN+s}{v}(-1)^{uN+l-v} \int^1_0\xi^{\tfrac{1+l}{u}+v-1}(1-\xi)^{\tfrac{u-1}{u}-1}d\xi \label{Gamma rep}\\
&=\dfrac{(-1)^{uN+s}\Gamma(\tfrac{u+l}{u})}{\Gamma(\tfrac{1+l}{u})\Gamma(\tfrac{u-1}{u})}\int^1_0\xi^{\tfrac{1+l}{u}-1}(1-\xi)^{uN+s+\tfrac{u-1}{u}-1}d\xi \nonumber\\
&=\dfrac{(-1)^{uN+s}\Gamma(\tfrac{u+l}{u})}{\Gamma(uN+s+\tfrac{u+l}{u})}\dfrac{\Gamma(uN+s+\tfrac{u-1}{u})}{\Gamma(\tfrac{u-1}{u})} \nonumber\\
&=\dfrac{(-1)^{uN+s}\left(\tfrac{u-1}{u}\right)_{uN+s}}{\left(\tfrac{u+l}{u}\right)_{uN+s}}. \nonumber
\end{align} 
Substituting above equality into Equation $(\ref{Theta})$, we obtain the assertion.
\end{proof}

\subsection{Estimates}
Unless stated otherwise, the Landau symbols small $o$ and large $O$ refer when $N$ tends to infinity.

For a finite set $S$ of rational numbers and a rational number $a$, we define $${\rm{den}}\,(S)=\min\{n\in \Z\mid n\ge 1, ns\in \Z \ \text{for all} \ s\in S\} \ \ \ \text{and}  \ \ \ \mu(a)={\rm{den}}\,(a)\prod_{\substack{q:{\rm{prime}}\\ q|{\rm{den}}(a)}} q^{1/(q-1)}.$$
We now quote an estimate of the denominator of $((a)_k/(b)_k)_{0\le k \le n}$ for $n\in \N$ and $a,b\in \Q$ being nonnegative integers.
\begin{lemma} \label{valuation}  {\rm{\cite[Lemma $5.1$]{KP}}}
Let $n\in \N$ and $a,b\in \Q$ being nonnegative integers. 
Put $$D_n={\rm{den}}\left(\dfrac{(a)_0}{(b)_0},\ldots,\dfrac{(a)_n}{(b)_n}\right).$$
Then we have
$$\limsup_{n\to \infty}\dfrac{1}{n}\log\,D_n\le \log\,\mu(a)+\dfrac{{\rm{den}}(b)}{\varphi({\rm{den}}(b))},$$
where $\varphi$ is Euler's totient function.
\end{lemma}
{{For a rational number $a$ and a nonnegative integer $b$, we denote $\binom{a}{b}=(-1)^k(-a)_b/b!$.}}
\begin{lemma} \label{P,Q,R}
Let $N,l,h$ be nonnegative integers with $0\le l \le u-2$ and $0\le h\le u-1$.

$(i)$ We have 
$$P_{uN,h}(z)=(-1)^{uN}\sum_{k=0}^{N(u-1)}\left[\sum_{s=0}^k\binom{uN-1/u}{s+N}\binom{u(s+N)+h}{uN}\binom{1/u}{k-s}\right](-1)^kz^{uk+h}.$$

$(ii)$ Put $\tilde{\varepsilon}_{l,h}= 1$ if $h<l+1$ and $0$ if $l+1\le h$.
We have 
{\small{\begin{align*}
&Q_{uN,l,h}(z)=(-1)^{uN} \sum_{v=\tilde{\varepsilon}_{l,h}}^{N(u-1)}\left(\sum_{k=0}^{(u-1)N-v}(-1)^{k+v}\left[\sum_{s=0}^{k+v}\binom{uN-1/u}{s+N}\binom{u(s+N)+h}{uN} \binom{1/u}{k+v-s}\right] \dfrac{\left(\tfrac{1+l}{u}\right)_{k}}{\left(\tfrac{u+l}{u}\right)_{k}}\right)z^{uv+h-l-1}.
\end{align*}}}

$(iii)$ Put $\varepsilon_{l,h}=1$ if $l< h$ and $\varepsilon_{l,h}=0$ if $h\le l$.
We have 
$$\mathfrak{R}_{uN,l,h}(z)=\dfrac{\left(\tfrac{u-1}{u}\right)_{uN}}{\left(\tfrac{u+l}{u}\right)_{uN}z^{uN+l-h+1}}{\displaystyle{\sum_{k=\varepsilon_{l,h}}^{\infty}}}\binom{u(N+k)+l-h}{uN}\dfrac{\left(\tfrac{1+l}{u}\right)_{k}}{\left(\tfrac{u+l}{u}+uN\right)_{k}}\dfrac{1}{z^{uk}}.
$$
\end{lemma}
\begin{proof}
$(i)$ 
Put $$w(z)=(1-z^u)^{-1/u}=\sum_{k=0}^{\infty}\binom{-1/u}{k}(-z^u)^{k}\in K[[z]].$$
Then $w(z)$ is a solution of $-(z^u-1)\partial_z-z^{u-1}\in K[z,\partial_z]$. Lemma $\ref{decompose}$ yields 
$$
\dfrac{1}{(uN)!}\left(\partial_z-\dfrac{z^{u-1}}{z^u-1}\right)^{uN}(z^u-1)^{uN}=\dfrac{1}{(uN)!}w(z)^{-1}\partial^{uN}_zw(z)(z^u-1)^{uN},
$$
and therefore
\begin{align*}
P_{uN,h}(z)&=\dfrac{1}{(uN)!}w(z)^{-1}\partial^{uN}_z w(z) (z^u-1)^{uN}\cdot z^h\\
              &=\dfrac{(-1)^{uN}}{(uN)!}w(z)^{-1}\partial^{uN}_z\cdot \sum_{k=0}^{\infty}\binom{uN-1/u}{k}(-1)^{k}z^{uk+h}\\
&={(-1)^{uN}}\sum_{k=0}^{\infty}\binom{1/u}{k}(-1)^kz^{uk} \cdot \sum_{k=0}^{\infty}\binom{uN-1/u}{k+N}\binom{u(k+N)+h}{uN}(-1)^{k}z^{uk+h}\\
&=(-1)^{uN}\sum_{k=0}^{\infty}\left[\sum_{s=0}^k\binom{uN-1/u}{s+N}\binom{u(s+N)+h}{uN} \binom{1/u}{k-s}\right](-1)^kz^{uk+h}.
\end{align*}
Since ${\rm{deg}}\, P_{uN,h}=u(u-1)N+h$, using above equality, we obtain the assertion.

\bigskip

$(ii)$ Put $P_{uN,h}(z)=\sum_{k=0}^{u(u-1)N+h}p_kz^{k}$. Notice that, by item $(i)$, 
$$p_k=
\begin{cases}
(-1)^{uN+k'}\sum_{s=0}^{k'}\binom{uN-1/u}{s+N}\binom{u(s+N)+h}{uN} \binom{1/u}{k'-s}   & \ \ \text{if there exists} \ k'\ge 0 \  \text{such that} \ k=uk'+h,\\
0 & \ \ \text{otherwise} .
\end{cases}
$$
Then we have 
\begin{align*}
\dfrac{P_{uN,h}(z)-P_{uN,h}(t)}{z-t}&=\sum_{k'=1}^{u(u-1)N+h}p_{k'}\sum_{v'=0}^{k'-1}z^{v'}t^{k'-v'-1}=\sum_{k'=0}^{u(u-1)N+h-1}p_{k'+1}\sum_{v'=0}^{k'}z^{v'}t^{k'-v'}\\
&=\sum_{v'=0}^{u(u-1)N+h-1}\left[\sum_{k'=v'}^{u(u-1)N+h-1}p_{k'+1}t^{k'-v'}\right]z^{v'}\\
&=\sum_{v'=0}^{u(u-1)N+h-1}\left[\sum_{k'=0}^{u(u-1)N+h-v'-1}p_{k'+v'+1}t^{k'}\right]z^{v'}.
\end{align*}
Since $\varphi_l(t^{k'})=0$ if $k' \not\equiv l \ {\rm{mod}} \ u$, putting $k'=uk+l$, we obtain
\begin{align*}
Q_{uN,l,h}(z)&=\varphi_l\left(\dfrac{P_{uN,h}(z)-P_{uN,h}(t)}{z-t}\right)\\
&=\sum_{v'=0}^{u(u-1)N+h-1}\left[\sum_{k'=0}^{u(u-1)N+h-v'-1}p_{k'+v'+1}\varphi_l(t^{k'})\right]z^{v'}\\
&=\sum_{v'=0}^{u(u-1)N+h-1}\left[\sum_{k=0}^{(u-1)N+\lfloor (h-v'-l-1)/u\rfloor}p_{uk+l+v'+1}\dfrac{\left(\tfrac{1+l}{u}\right)_k}{\left(\tfrac{u+l}{u}\right)_k}\right]z^{v'}.
\end{align*}
Since we have $p_{uk+l+v'+1}=0$ for $0\le v'\le u(u-1)N+h-1$ with $v'\notin u\Z+h-l-1$, putting $v'=uv+h-l-1$, 
we conclude
\begin{align*}
Q_{uN,l,h}(z)&=\sum_{v=\tilde{\varepsilon}_{l,h}}^{(u-1)N}\left[\sum_{k=0}^{(u-1)N-v}p_{u(k+v)+h}\dfrac{\left(\tfrac{1+l}{u}\right)_k}{\left(\tfrac{u+l}{u}\right)_k}\right]
z^{uv+h-l-1}.
\end{align*}
This completes the proof of item $(ii)$.

\bigskip

$(iii)$ Lemma $\ref{coeff n+1}$ yields 
\begin{align} \label{RuNlh}
\mathfrak{R}_{uN,l,h}(z)=\sum_{k=uN}^{\infty}\dfrac{\varphi_l(t^kP_{uN,h}(t))}{z^{k+1}}.
\end{align}
We now compute $\varphi_l(t^kP_{uN,h}(t))$ for $k\ge uN$.
Put $\mathcal{E}=\partial_t-t^{u-1}/(t^u-1)$. 
Using Proposition $\ref{key 1 2}$ $(i)$ for $k\ge uN$, there exists a set of integers $\{c_{uN,k,v}\mid v=0,1,\ldots ,uN\}$ with 
\begin{align*}
&c_{uN,k,uN}=(-1)^{uN}k(k-1)\cdots (k-uN+1) \ \ \text{and} \\
&t^k \mathcal{E}^{uN}(t^u-1)^{uN}=\sum_{v=0}^{uN}c_{uN,k,v}\mathcal{E}^{uN-v}t^{k-v}(t^u-1)^{uN} \ \ \text{in} \ \ \Q(t)[\partial_t].
\end{align*}
Since $\mathcal{E}(t^u-1)\subseteq  {\rm{ker}}\,\varphi_l$, using above relation,
\begin{align}
\varphi_l(t^kP_{uN,h}(t))&=\varphi_l\left(\dfrac{t^k}{(uN)!}\mathcal{E}^{uN}(t^u-1)^{uN}\cdot t^h\right)=\varphi_l\left(\sum_{v=0}^{uN}\dfrac{c_{uN,k,v}}{(uN)!}\mathcal{E}^{uN-v}t^{k-v}(t^u-1)^{uN}\cdot t^h\right) \label{tkP}\\
&=\varphi_l\left(\dfrac{c_{uN,k,uN}}{(uN)!}t^{k-uN}(t^u-1)^{uN}\cdot t^h\right)=(-1)^{uN}\binom{k}{uN}\varphi_l(t^{k-uN+h}(t^u-1)^{uN}). \nonumber
\end{align}
Note we have $\varphi_l(t^{k-uN+h}(t^u-1)^{uN})=0$ if $k-uN+h \not\equiv l \ {\rm{mod}} \ u$.
Let $\tilde{k} \ge 0$ and put $k=u(\tilde{k}+N+\varepsilon_{l,h})+l-h$. A similar computation which we performed in Equation $(\ref{Gamma rep})$ implies  
\begin{align*}
\varphi_l(t^{k-uN+h}(t^u-1)^{uN})&=\varphi_l(t^{u(\tilde{k}+\varepsilon_{l,h})+l}(t^u-1)^{uN})\\
&=\dfrac{(-1)^{uN}\left(\tfrac{u-1}{u}\right)_{uN} \left(\tfrac{1+l}{u}\right)_{\tilde{k}+\varepsilon_{l,h}}}{\left(\tfrac{u+l}{u}\right)_{uN+\tilde{k}+\varepsilon_{l,h}}}=
\dfrac{(-1)^{uN}\left(\tfrac{u-1}{u}\right)_{uN} \left(\tfrac{1+l}{u}\right)_{\tilde{k}+\varepsilon_{l,h}}}{\left(\tfrac{u+l}{u}\right)_{uN}\left(\tfrac{u+l}{u}+uN\right)_{\tilde{k}+\varepsilon_{l,h}}}
.
\end{align*}
Substituting above equality into Equations \eqref{tkP} and $(\ref{RuNlh})$, we obtain the desire equality.
\end{proof}
In the following, for a rational number $a$ and a nonnegative integer $n$, we put 
$$\mu_n(a)={\rm{den}}(a)^n\prod_{\substack{q:\mathrm{prime} \\ q|{\rm{den}}(a)}}q^{\left\lfloor \tfrac{n}{q-1}\right\rfloor} .$$
Notice that $\mu_n(a)=\mu_n(a+k)$ for $k\in \Z$ and
\begin{align} \label{divide}
\mu_{n_2}(a) \ \ \text{is divisible by} \ \ \mu_{n_1}(a) \ \ \ \text{and} \ \ \ \mu_{n_1+n_2}(a) \ \ \text{is divisible by} \ \ \mu_{n_1}(a)\mu_{n_2}(a)
\end{align}
for $n,n_1,n_2\in \N$ with $n_1\le n_2$.
\begin{lemma} \label{P,Q}
Let $K$ be an algebraic number field, $v$ a place of $K$ and $\alpha\in K\setminus\{0\}$. 

$(i)$ We have 
$$\max_{0\le h \le u-1}\log\,|P_{uN,h}(\alpha)|_{v}\le o(N)+u(u-1){\rm{h}}_{v}(\alpha)N+
\begin{cases}
u(u+1){\rm{h}}_v(2)N & \ \text{if} \ v\mid \infty\\ 
\log\,|\mu_{uN}(1/u)|^{-1}_{v} & \ \text{if} \ v \nmid \infty.
\end{cases}$$
({The function $o(N)$ is equal to $0$ for almost all places $v$. This also holds in the statement $(ii)$.})

$(ii)$ For $0\le l \le u-2$, put $$D_{N}={\rm{den}}\left(\dfrac{\left(\tfrac{1+l}{u}\right)_k}{\left(\tfrac{u+l}{u}\right)_k}\right)_{\substack{0\le l \le u-2 \\ 0 \le k \le (u-1)N}}.$$
Then,
{\small{\begin{align*}
\max_{\substack{0\le l \le u-2 \\ 0\le h \le u-1}}\log\,|Q_{uN,l,h}(\alpha)|_{v}&\le o(N)+u(u-1){\rm{h}}_v(\alpha)N+
\begin{cases}
u(u+1){\rm{h}}_v(2)N  & \ \text{if} \ v \mid \infty\\
\log\,|\mu_{uN}(1/u)|^{-1}_{v}+\log\,|D_{N}|^{-1}_v & \ \text{if} \ v \nmid \infty.
\end{cases}
\end{align*}}}
\end{lemma}
\begin{proof}
$(i)$ Let $v$ be an archimedean place. 
Since we have 
\begin{align*}
\binom{uN-1/u}{s+N}
\le 2^{uN}, \ \  \ \binom{u(s+N)+h}{uN}\le 2^{u(s+N)+h} \ \ \text{and} \ \ \left|\binom{1/u}{k-s}\right|\le 1 ,
\end{align*}
for $0\le k \le N(u-1)$ and $0\le s \le k$, we obtain 
\begin{align} \label{trivial estimates}
\left|\sum_{s=0}^k\binom{uN-1/u}{s+N}\binom{u(s+N)+h}{uN}\binom{1/u}{k-s}\right|\le 2^{2uN+h}\sum_{s=0}^k2^{us}\le 2^{2uN+h+u(k+1)}. 
\end{align}
Thus, by Lemma $\ref{P,Q,R}$ $(i)$, 
\begin{align*}
|P_{uN,h}(\alpha)|_{v}\le |2^{2uN+h}|_{v}\cdot \left|\sum_{k=0}^{N(u-1)}2^{u(k+1)}\alpha^{uk+h}\right|_{v}\le e^{o(N)}|2|^{u(u+1)N}_v \max\{1,|\alpha|_v\}^{u(u-1)N}.
\end{align*}
This completes the proof of the archimedean case.

Second, we consider the case of $v$ is a nonarchimedean place. 
Note that 
$$\binom{uN-1/u}{s+N}=\dfrac{(-1)^{s+N}(1/u-uN)_{s+N}}{(s+N)!} \ \ \ \text{and} \ \ \ \binom{1/u}{k-s}=\dfrac{(-1)^{k-s}(-1/u)_{k-s}}{(k-s)!}$$
for $0\le k \le N(u-1), \ 0\le s \le k$. 
Combining $$\left|\dfrac{(a)_k}{k!}\right|_v\le |\mu_n(a)|^{-1}_v \ \ \text{for} \ \ a\in \Q \ \ \text{and} \ \ k,n\in \N \ \ \text{with} \ \ k\le n,$$
(see \cite[Lemma $2.2$]{B}) 
and \eqref{divide} yields 
$$\left|\binom{uN-1/u}{s+N}\binom{u(s+N)+h}{uN} \binom{1/u}{k-s}\right|_v\le |\mu_{k+N}(1/u)|^{-1}_v \ \ \ \text{for} \ \ \ 0\le k \le (u-1)N.$$
Therefore the strong triangle inequality yields
\begin{align} \label{trivial 2}
\max_{0\le k \le N(u-1)}\left|\sum_{s=0}^k\binom{uN-1/u}{s+N}\binom{u(s+N)+h}{uN} \binom{1/u}{k-s}\right|_v\le |\mu_{uN}(1/u)|^{-1}_v.
\end{align}
Use Lemma $\ref{P,Q,R}$ $(i)$ again, we conclude the desire inequality. 

\bigskip

$(ii)$
Let $v$ be an archimedean place. We use the same notation as in the proof of Lemma $\ref{P,Q,R}$ $(ii)$.
Using Equation $(\ref{trivial estimates})$ again, we obtain 
\begin{align*}
\left|\sum_{k=0}^{(u-1)N-v}p_{u(k+v)+h}\dfrac{\left(\tfrac{1+l}{u}\right)_k}{\left(\tfrac{u+l}{u}\right)_k}\right|_v&\le |2|^{2uN+h+u(v+1)}_{v}\sum_{k=0}^{N(u-1)-v}|2|^{uk}_v\\
&\le |2|^{2uN+u(u-1)N+u+h}_{v}.
\end{align*}
Lemma $\ref{P,Q,R}$ $(ii)$ implies that
\begin{align*}
|Q_{uN,l,h}(\alpha)|_v&\le 
\sum_{v=0}^{(u-1)N}\left|\left[\sum_{k=0}^{(u-1)N-v}p_{u(k+v)+h}\dfrac{\left(\tfrac{1+l}{u}\right)_k}{\left(\tfrac{u+l}{u}\right)_k}\right]\right|_v
|\alpha|^{uv+h-l-1}_v\\
&\le e^{o(N)}|2|^{u(u+1)N}_v \max\{1,|\alpha|_v\}^{u(u-1)N}.
\end{align*}

Let $v$ be a nonarchimedean place. 
Then by the definition of $D_N$, 
\begin{align*}
\max_{\substack{0 \le l \le u-2 \\ 0 \le k \le (u-1)N}}\left(\left| \dfrac{\left(\tfrac{1+l}{u}\right)_k}{\left(\tfrac{u+l}{u}\right)_k}\right|_v\right)\le |D_N|^{-1}_v,
\end{align*}
for all $N\in \N$.
Using above inequality and $(\ref{trivial 2})$ for Lemma $\ref{P,Q,R}$ $(ii)$, we obtain the desire inequality by the strong triangle inequality.
This completes the proof of Lemma $\ref{P,Q}$.
\end{proof}

\begin{lemma} \label{R}
Let $K$ be an algebraic number field, $v_0$ a place of $K$, $\alpha\in K$. Let $N,l,h$ be nonnegative integers with $0\le l \le u-2$ and $0\le h \le u-1$.

$(i)$ Assume $v_0$ is an archimedean place and $|\alpha|_{v_0}>2$. We have 
$$\max_{\substack{0\le l \le u-2 \\ 0\le h \le u-1}} \log\,|\mathfrak{R}_{uN,l,h}(\alpha)|_{v_0} \le -u({\rm{h}}_{v_0}(\alpha)-{\rm{h}}_{v_0}(2))N+o(N) .$$

$(ii)$ Assume $v_0$ is a nonarchimedean place and $|\alpha|_{v_0}>1$. Let $p_{v_0}$ be the rational prime under $v_0$. 
Put $\varepsilon_{v_0}(u)=1$ if $u$ is coprime with $p_{v_0}$ and $\varepsilon_{v_0}(u)=0$ if $u$ is divisible by $p_{v_0}$.
We have 
$$
\max_{\substack{0\le l \le u-2 \\ 0\le h \le u-1}} \log\,|\mathfrak{R}_{uN,l,h}(\alpha)|_{v_0} \le -u\left({\rm{h}}_{v_0}(\alpha)-\dfrac{\varepsilon_{{v_0}}(u) \log\,|p_{v_0}|_{v_0}}{p_{v_0}-1}\right)N+o(N) .
$$
\end{lemma}
\begin{proof}
$(i)$ For a nonnegative integer $k$, we have $\binom{u(N+k)+l-h}{uN}\le 2^{u(N+k)+l-h}$.
Thus,
{\small{\begin{align*}
\left|{\displaystyle{\sum_{k=\varepsilon_{l,h}}^{\infty}}}\binom{u(N+k)+l-h}{uN}\dfrac{\left(\tfrac{1+l}{u}\right)_{k}}{\left(\tfrac{u+l}{u}+uN\right)_{k}}\dfrac{1}{\alpha^{uk}}\right|_{v_0}
&\le|2^{uN+l-h}|_{v_0}{\displaystyle{\sum_{k=\varepsilon_{l,h}}^{\infty}}}\left|\dfrac{\left(\tfrac{1+l}{u}\right)_{k}}{\left(\tfrac{u+l}{u}+uN\right)_{k}}\right|_{v_0}\left|\dfrac{2}{\alpha}\right|^{uk}_{v_0}\\
&\le |2^{uN+l-h}|_{v_0}{\displaystyle{\sum_{k=0}^{\infty}}}\left|\dfrac{2}{\alpha}\right|^{uk}_{v_0}= |2^{uN}|_{v_0}e^{o(N)}.
\end{align*}}}
Using above inequality in Lemma $\ref{P,Q,R}$ $(ii)$, we obtain the assertion.

$(ii)$ 
By {\rm{\cite[Proposition $4$,\, Lemma $4$]{DRR} (loc. cit. $(6.1)$, $(6.2)$)}}, 
\begin{align*}
&\max_{0\le l \le u-2}\, \left(\left| \dfrac{\left(\tfrac{u-1}{u}\right)_{uN}}{\left(\tfrac{u+l}{u}\right)_{uN}} \right|_{v_0} \right) \le 
|p_{v_0}|^{\varepsilon_{{{v_0}}}(u)v_{p_{v_0}}((uN)!)+o(N)}_{v_0},\\
&\left|\displaystyle{\sum_{k=\varepsilon_{l,h}}^{\infty}}\binom{u(N+k)+l-h}{uN}\dfrac{\left(\tfrac{1+l}{u}\right)_{k}}{\left(\tfrac{u+l}{u}+uN\right)_{k}}\dfrac{1}{\alpha^{uk}}\right|_v=e^{o(1)}.
\end{align*}
Combining $v_p((uN)!)=uN/(p-1)+o(N)$ and above inequality in Lemma $\ref{P,Q,R}$ $(ii)$, we obtain the assertion.
This completes the proof of Lemma $\ref{R}$.
\end{proof}

\subsection{Proof of Theorem $\ref{special hypergeometric}$}
\begin{proof}
We use the same notation as in Theorem $\ref{special hypergeometric}$.
Let $\alpha\in K$ with $|\alpha|_{v_0}>1$.
For a nonnegative integer $N$, we define a matrix 
$${\rm{M}}_N=
\begin{pmatrix} 
P_{uN,0}(\alpha) & \cdots & P_{uN,u-1}(\alpha)\\ 
Q_{uN,0,0} (\alpha) & \cdots & Q_{uN,0,u-1} (\alpha)\\
\vdots & \ddots & \vdots\\
Q_{uN,u-2,0} (\alpha) & \cdots & Q_{uN,u-2,u-1} (\alpha)
\end{pmatrix}\in {\rm{M}}_u(K)
.$$ 
By Lemma $\ref{non vanish uN}$, the matrices ${\rm{M}}_N$ are invertible for every $N$. 
We define functions 
$$F_v:\N\longrightarrow \R_{\ge0}; \ N\mapsto  u(u-1){\rm{h}}_v(\alpha)N
+o(N)+
\begin{cases}
u(u+1){\rm{h}}_v(2)N  & \ \text{if} \ v \mid \infty\\
\log\,|\mu_{uN}(1/u)|^{-1}_{v}+\log\,|D_N|^{-1}_v & \ \text{if} \ v \nmid \infty,
\end{cases}
$$
for $v\in \mathfrak{M}_K$. By Lemma $\ref{valuation}$, 
\begin{align*}
&\lim_{N\to \infty}\dfrac{1}{N}\log\,D_N\le (u-1)\left(\log\,\nu(u)+\dfrac{u}{\varphi(u)}\right),
\end{align*} where $D_N$ is the integer defined in Lemma $\ref{P,Q}$,  
\begin{align*}
&\lim_{N\to \infty}\dfrac{1}{N}\left(\sum_{v\neq v_0}F_v(N)\right)\le u\mathbb{B}_{v_0}(\alpha),
\end{align*}
and, by Lemma $\ref{P,Q}$,
\begin{align*}
&\max_{\substack{0\le h \le u-1}}\log\, \max\{|P_{uN,h}(\alpha)|_{v_0}\}\le u U_{v_0}(\alpha)N+o(N) ,\\
&\max_{\substack{0\le l \le u-2 \\ 0\le h \le u-1}}\log\, \max\{|P_{uN,h}(\alpha)|_v, |Q_{uN,l,h}(\alpha)|_{v}\}\le F_{v}(N) \ \ \text{for} \ \ v\in \mathfrak{M}_K.
\end{align*}
By Lemma $\ref{R}$, 
\begin{align*}
\max_{\substack{0 \le l \le u-2 \\ 0\le h \le u-1}} \log\, |\mathfrak{R}_{uN,l,h}(\alpha)|_{v_0}\le -u  \mathbb{A}_{v_0}(\alpha)N+o(N) .
\end{align*}
Using a {{quantitative}} linear independence criterion in \cite[Proposition $5.6$]{DHK2} for 
$$
\theta_l=\dfrac{1}{\alpha^{l+1}} {}_2F_1\left(\frac{1+l}{u},1, \left.\frac{u+l}{u}\right|\frac{1}{\alpha^u}\right) \ \ \text{for} \ \ 0\le l \le u-2,
$$ 
and the invertible matrices $({\rm{M}}_N)_N$, and applying above estimates, we obtain Theorem $\ref{special hypergeometric}$.
\end{proof}

{{\section{Appendix}
Denote the algebraic closure of $\Q$ by $\overline{\Q}$. Let $a(z),b(z)\in \overline{\Q}[z]$ with $w:=\max\{{\rm{deg}}\,a-2,{\rm{deg}}\,b-1\}\ge 0$ and $a(z)\neq 0$. Put $D=-a(z)\partial_z+b(z)$.
The Laurent series $f_0(z),\ldots,f_w(z)$ obtained in Lemma $\ref{solutions}$ for $D$ become $G$-functions in the sense of Siegel when $D$ is a $G$-operator (see \cite[IV]{An1}).
Here we refer below a result due to S.~Fischler and T.~Rivoal \cite{F-R} in which they
gave a condition so that $D$ becomes a $G$-operator.

\begin{lemma} $($\rm{cf. \cite[Proposition $3$ (ii)]{F-R}}$)$
Let $m\ge 2$ be a positive integer, $\alpha_1,\ldots,\alpha_m,\beta_1,\ldots,\beta_{m-1}, \gamma\in \overline{\Q}$ with $\alpha_1,\ldots, \alpha_{m}$ being pairwise distinct.
In the case of $0\in \{\alpha_1,\ldots,\alpha_m\}$, we put $\alpha_m=0$.
Define $a(z)=\prod_{i=1}^m(z-\alpha_i), b(z)=\gamma\prod_{j=1}^{m-1}(z-\beta_j)$ and $D=-a(z)\partial_z+b(z)\in \overline{\Q}[z,\partial_z]$.  
Then the following are equivalent. 

$(i)$ $D$ is a $G$-operator.

$(ii)$ We have 
\begin{align*} 
&\gamma\dfrac{\prod_{j=1}^{m-1}(\alpha_i-\beta_j)}{\prod_{i'\neq i}(\alpha_i-\alpha_{i'})}\in \Q \ \ \ \text{for all} \ \ \ 1\le i \le m \ \ \  \text{if} \ \ 0\notin \{\alpha_1,\ldots,\alpha_m\},\\
&\gamma\dfrac{\prod_{j=1}^{m-1}(\alpha_i-\beta_j)}{\prod_{i'\neq i}(\alpha_i-\alpha_{i'})}\in \Q \ \ \ \text{for all} \ \ \ 1\le i \le m \ \ \text{and} \ \ \gamma \prod_{j=1}^{m-1}\dfrac{\beta_j}{\alpha_j}\in \Q  \ \ \ \text{otherwise}. 
\end{align*}
\end{lemma}
}}

\noindent
{\bf Acknowledgements.}

The author is grateful to Professors Daniel Bertrand and Sinnou David for their helpful suggestions.
The author deeply thanks Professor Noriko Hirata-Kohno for her enlightening comments on a preliminary version.
{{The author also deeply thanks the anonymous referee for the careful reading and the helpful comments on earlier version of the article that helped us to improve it in various aspects.}}
This work is partly supported by the Research Institute for Mathematical Sciences, an international joint usage
and research center located in Kyoto University.

\bibliography{}

\medskip\vglue5pt
\vskip 0pt plus 1fill
\hbox{\vbox{\hbox{Makoto \textsc{Kawashima}}
\hbox{{\tt kawashima.makoto@nihon-u.ac.jp}}
\hbox{Department of Liberal Arts and Basic Sciences}
\hbox{College of Industrial Engineering}
\hbox{Nihon University}
\hbox{Izumi-chou, Narashino, Chiba}
\hbox{275-8575, Japan}
}}
\end{document}